\documentclass[3p,times]{elsarticle}
\usepackage[colorlinks,linkcolor=blue,citecolor=blue]{hyperref}
\usepackage{bookmark}
\usepackage{bm}
\usepackage{amsmath,amsfonts,amsmath,amssymb,amsthm,cases}
\usepackage{algorithm}
\usepackage{algorithmicx}
\usepackage{algpseudocode}
\usepackage{multirow}
\usepackage{longtable}
\usepackage{fullpage}
\usepackage{cleveref}
\usepackage{graphics,graphicx,epsfig,epstopdf,subfigure}
\usepackage{color,natbib}
\usepackage{CJK}
\graphicspath{{jpg/}}

\newtheorem{theorem}{Theorem}[section]
\newtheorem{lemma}{Lemma}[section]

\newtheorem{remark}{Remark}[section]
\newtheorem{example}{Example}[section]

\newtheorem{assumption}{Assumption}[section]
\numberwithin{equation}{section}

\newcommand{\dd}{\,{\rm d}}

\newcommand{\ifft}{\,{\rm ifft}}

\begin{document}
\begin{CJK}{UTF8}{gkai}
\begin{frontmatter}
	\title{Exponential integrator Fourier Galerkin methods for semilinear parabolic equations}%% \tnoteref{mytitlenote}}
	\author[sjtu]{Jianguo Huang\fnref{hjgfootnote}}
	\ead{jghuang@sjtu.edu.cn}
	\address[sjtu]{School of Mathematical Sciences, and MOE-LSC, Shanghai Jiao Tong University, Shanghai 200240, China}
	%\fntext[hjgfootnote]{J. Huang's research was partially supported by National Natural Science Foundation of China under grant number  11571237.}
	\author[psu]{Yuejin Xu\fnref{xyjfootnote}}
	\ead{ymx5204@psu.edu}
	\address[psu]{Department of mathematics, the Pennsylvania State University, PA 16802, USA}
	\begin{abstract}
		In this paper, in order to improve the spatial accuracy, the exponential integrator Fourier Galerkin method (EIFG) is proposed for solving semilinear parabolic equations in rectangular domains. In this proposed method, the spatial discretization is first carried out by the Fourier-based Galerkin approximation, and then the time integration of the resulting semi-discrete system is approximated by the explicit exponential Runge-Kutta approach, which leads to the fully-discrete numerical solution. With certain regularity assumptions on the model problem, error estimate measured in $H^2$-norm is explicitly derived for EIFG method with two RK stages. Several two and three dimensional examples are shown to demonstrate the excellent performance of EIFG method, which are coincident to the theoretical results. 
	\end{abstract}
\end{frontmatter}	
	
\section{Introduction}

The spectral method has become increasing popular in numerical solutions of partial differential equations due to its high-order accuracy, see \cite{CialetLions1997, CanutoHussaini1988, Guo1998, ShenTang2011, CanutoHussaini2006, GottliebHussaini1984, ShenTang2006, Shen1994, HesthavenGottlieb2007}. In the earlier work, the spatial variables are discretized by the spectral method while the temporal variables are usually applied with finite difference scheme for time-dependent partial differential equations. For instance, for phase field equations \cite{ChenShen1998}, linearized Naiver-Stokes equations \cite{GervasioSaleri1998} and hyperbolic equations \cite{Tal1986}, temporal variables are discretized by semi-implicit scheme. Semi-implicit schemes allow much larger time step sizes than explicit schemes \cite{ChenShen1998} while maintaining higher accuracy without increasing the computational work and memory space \cite{Tal1986}. Moreover, the efficiency and stability of the numerical schemes can be improved by some stabilization techniques \cite{GervasioSaleri1998}, which enable us to develop better solvers for stiff problems. However, when the exact solution is sufficiently smooth, the accuracy of the numrical solution would be limited by the accuracy of finite difference scheme in time. In recent years, the spectral method and finite element method in time has been applied to temporal discretization. As for the finite element temporal discretization, Shen Jie and Wang Lilian have proposed a new space-time spectral method based on a Legendre-Galerkin method in space and a dual-Petrov-Galerkin formulation in time \cite{ShenWang2007}. The method is of great efficiency due to the Fourier-like basis functions are orthogonal with respect to $L^2$ and $H^1$-norms. Moreover, the method combined with single or multi-interval Legendre Petrov-Galerkin method in time is presented in \cite{TangMa2006} and the corresponding optimal error estimate in $L^2$-norm is derived. The h-p version of the finite element method for both time and space variables is also addressed in \cite{BabuvskaJanik1989} with eigenvalue decomposition. As for the spectral temporal discretization, Hillel Tal-Ezer \cite{Tal1989} has presented a pseudospectral explicit scheme for solving linear, periodic problems, which has infinite accuracy both in time and in space. For solving nonlinear problems, temporal single-interval and multi-interval Legendre-Gauss-Radau collocation method is developed in \cite{WangTong2024}. In the past few years, many scholars have extensively explored the application of spectral method and applied it to various problems, such as singular eigenvalue problems \cite{MaLi2018}, fractional Laplacian problems \cite{LinAza2018, HouXu2017, ShengCao2021, TangWang2020, ZakyHendy2020, AlzahraniKhaliq2019}, Allen-Cahn problems \cite{JiaZhang2021}, high-order problems \cite{ZhangYu2021}, mixed inhomogeneous boundary value problems \cite{YuGuo2014}, quadrilateral domains \cite{YuGuo2016}, triangle domains \cite{SamsonLi2013}, and so on.

Furthermore, based on the Jacobi interpolation approximations \cite{GuoWang2004}, Jacobi-weighted Sobolev spaces have been constructed and the spectral method can also be applied to equations with Dirichlet boundary condition \cite{Guo1998,GuoShen2006,GuoWang2001}. The error analysis for mapped Jacobi spectral methods and Fourier-Legendre and Fourier-Hermite spectral-Galerkin methods has been derived in \cite{WangShen2005,ZhangWang2023}. Besides Legendre interpolations, spatial approximation is always realized by Chebyshev \cite{Shen1995, PetterPeter1997} or Hermite interpolations \cite{AguirreRivas2005, MaSun2005, MaoShen2017}. Meanwhile, Legendre-based and Chebyshev-based spectral method can also be applied to hyperbolic equations \cite{ShenWang2007b}.

Exponential time differencing (ETD) method \cite{CoxMatthews2002, BorislavWill2005} is an exponential integrator-based method and has been widely used in engineering and scientific computing due to its great efficiency and stability in handling stiff semilinear systems. ETD method first approximates the nonlinear part using polynomial interpolation and then performs exact integration on the resulting integrands. The evaluations of the products of matrix exponentials and vectors are usually implemented with Krylov subspace method \cite{DuJu2021}, and some researches in \cite{DuJu2019,ZhuJu2016,JuZhang2015} have tried to imporve the efficiency with various techniques. Also, the stability property in $L^2$ and $L^{\infty}$ norms have been established mainly with the help of maximum bound principle \cite{DuZhu2004}. Moreover, a series of consistency and order conditions were systematically developed for exponential Runge-Kutta method based on the theory of semigroup \cite{Pope1963,HochbruckOstermann2005b,HochbruckOstermann2005a,HochbruckOstermann2010,LUAN2014}. Additionally, the stability has been enhanced in \cite{DuZhu2004,KassamTrefethen2005,MohebbiDehghan2010} and the application areas have been broadened in \cite{HochbruckOstermann2008,HuangJu2019b,HuangJu2019a,RoSc2021,WhalenBrio2015}.

In this paper, we propose an efficient exponential integrator Fourier Galerkin method (EIFG) to solve the semilinear parabolic equation with periodic boundary conditions,
\begin{equation}
\label{eq1-1}
\begin{split}
u_t = \mathcal{D}\Delta u +f(t,u,\nabla u), \quad \bm x \in \Omega, \quad 0\leq t \leq T,
\end{split}
\end{equation}
where $\Omega$ is an open rectangular domain in $\mathbb{R}^d$ $(d \geq 1)$, $T > 0$ is the duration time, $\mathcal D>0$ is the diffusion coefficient, $u(t,\bm x)$ is the unknown function and $f(t,u,\nabla u)$ is the reaction term. EIFG method first discretes the spatial variables with Fourier-based Galerkin to obtain a semi-discrete (in space) system. Then it applies the explicit Runge-Kutta approach to approximate the time integration and achieves the fully-discrete scheme. The error estimate derived in $H^2$-norm is rigorous with hidden constants are independent of the spatial mesh size and temporal step size. In one aspect, the spatial accuracy can be obtained by discussing the accuracy of Fourier approximation, which is mainly depend on the regularity of exact solution. In another aspect, the temporal accuracy can be estimated by following the similar arguments in \cite{HochbruckOstermann2005b,HuangJu2023} since the fully-discrete system can be regarded as a finite dimensional evolution equation. To the best of our knowledge, the work presented in this paper is the first study on combining exponential integrator in time and Fourier Galerkin method in space, which provides explicit fully-discrete error estimates for general parabolic equations and improves the spatial accuracy significantly in \cite{HuangJu2023}.

The rest of the paper is organized as follows. The EIFG method is first proposed in Section \ref{algorithm-description}, and the fully-discrete error analysis is given in Section \ref{theory}. In Section \ref{numerical}, various numerical experiments are carried out to validate the theoretical results and demonstrate the excellent performance of the EIFG method.  Finally, some including remarks are drawn in Section \ref{conclusion}.

\section{The exponential integrator Fourier Galerkin method}\label{algorithm-description}

In this section, we first develop the exponential integrator Fourier Galerkin method (abbreviated as EIFG) for solving \eqref{eq1-1} with periodic boundary conditions. First of all, some standard notations are proposed for later provement. For a given bounded Lipschitz domain $\Omega \subset \mathbb{R}^d$ and nonnegative integer $s$, denote $H^s(\Omega)$ as the standard Sobolev spaces on domain $\Omega$ with norm $\|\cdot\|_{s,\Omega}$ and semi-norm $|\cdot|_{s,\Omega}$, and the corresponding $L^2$-inner product is $(\cdot,\cdot)_{\Omega}$. Also, we denote $H_p^s(\Omega)$ as the subspace of $H^s(\Omega)$, which consists of functions with derivatives of order up to $s-1$ being periodic. The corresponding norm of space $H_p^s(\Omega)$ is $\|\cdot\|_{s,\Omega}$ and $\|v\|_{k,\infty,\Omega}={\rm ess}\sup_{|\bm \alpha|\leq k}\|D^{\bm \alpha}v\|_{L^{\infty}(\Omega)}$ for any function $v$ such that the right-hand side term makes sense, where $\bm \alpha=(\alpha_1, \cdots, \alpha_d)$ is a multi-index and $|\bm \alpha|=\alpha_1+\cdots+\alpha_d$. Generally speaking, we omit the subscript for simplicity if there is no confusion. For a non-negative integer $\ell$, the set of all polynomials on $\Omega$ with the total degree at most $\ell$ are denoted as $P_{\ell}(\Omega)$. Moreover, given two quantities $a$ and $b$, ¡°$a \lesssim b$¡± is the abbreviation of ¡°$a \leq Cb$¡±, where the hidden constant $C$ is positive and independent of the mesh size; $a\eqsim b$ is equivalent to $a\lesssim b \lesssim a$.

Consider the rectangular domain $\Omega \in \mathbb{R}^d$ and assume $\bar{\Omega} := \prod\limits_{i=1}^d [a_i,b_i]$. Then we focus on the semilinear parabolic equation with periodic boundary condition and an initial configuration $u_0 \in H_p^m(\Omega)$, that is
\begin{equation}
\label{eq3-1}
\left\{\begin{split}
&u_t = \mathcal{D}\Delta u +f(t, u,\nabla u), \quad \bm x \in \Omega, \quad 0 \leq t \leq T,\\
&u(0, \bm x)=u_0(\bm x), \quad \bm x \in \Omega,\\
&u(t, \bm x)|_{x_i=a_i}= u(t,\bm x)|_{x_i=b_i}, \quad 1\leq i\leq d, \quad 0 \leq t \leq T,
\end{split}
\right.
\end{equation}
where $T \geq 0$ is the terminal time. For convenience, we always assume that $\mathcal{D}=1$.

\subsection{Semi-discretization in space by Fourier Galerkin approximation}

Given a positive even number $N_j$, we construct a one-dimensional Fourier approximation space for $[a_j,b_j]$ as
\begin{equation*}
X_{N_j}^j(a_j,b_j)={\rm span}\Big\{ \phi_j^k(x_j)\Big|-\frac{N_j}{2}\leq k \leq \frac{N_j}{2}-1 \Big\},
\end{equation*}
where $\phi_j^k(x_j)=e^{\rm ik\frac{b_j-a_j}{2\pi}x_j}$ is the $k$-th basis function of $X_{N_j}^j(a_j,b_j)$ and $\rm{i}=\sqrt{-1}$ is the imaginary number. By using the tensor product of all above Fourier approximation spaces, we can obtain the Fourier approximation space for $\Omega$ as follows:
\begin{equation}
\label{tensor_finite}
\begin{split}
X_N&:= X_{N_1}^1(a_1, b_1)\otimes \cdots \otimes X_{N_d}^d(a_d,b_d)\\
&= {\rm span} \Big\{\phi_1^{k_1}(x_1)\cdots \phi_d^{k_d}(x_d)\Big| -\frac{N_1}{2}\leq k_1 \leq \frac{N_1}{2}-1, \cdots, -\frac{N_d}{2}\leq k_d \leq \frac{N_d}{2}-1 \Big\}.
\end{split}
\end{equation}

It's obvious that $X_N \subset H_p^m(\Omega)$. Define $N=\max_{1\leq n \leq d} N_n$ as the number of spatial meshes of the corresponding uniformly rectangular partition $\mathcal{T}_N$ for generating $X_N$. For the forthcoming error analysis, we assume the order of polynomials in each direction is of the same magnitude, i.e., $N \eqsim N_n$ for $\forall \ 1 \leq n \leq d$. For brevity, denote $\Lambda_N=[-\frac {N_1} 2,\frac{N_1}{2}-1]\times\cdots\times[-\frac {N_d}2,\frac{N_d}{2}-1]$, ${\bm k}=(k_1, \cdots, k_d)$ and $\phi_{{\bm k}}(\bm x)=\phi_{k_1}^1(x_1)\cdots\phi_{k_d}^d(x_d)$, where $\bm k \in \Lambda_N$. Then the approximate function $u_N(\bm x)$ for $u(\bm x)$ can be regarded as the truncated Fourier series
\begin{equation}
\label{eigenvalues}
u_N(\bm x)=\sum\limits_{ {\bm k}\in \Lambda_N}\widehat{u}_{N,  {\bm k}}\phi_{ {\bm k}}(\bm x).
\end{equation}

Then by using the Fourier Galerkin formulation, the Fourier approximation in space for \eqref{eq3-1} is to find $u_N\in L^2(0,T;X_N)$ such that
\begin{equation}
\label{eq3-11}
\left\{\begin{split}
&(u_{N,t},v_N) + a(u_N,v_N) = (f(t,u_N,\nabla u_N),v_N), \quad \forall v_N \in X_N, 0 \leq t \leq T,\\
&u_N(0)=P_Nu_0,
\end{split}
\right.
\end{equation}
where the bilinear operator $a(\cdot, \cdot)$ is defined by
\begin{equation}\label{bilinear}
a(w_N,v_N)=\int_{\Omega}\mathcal{D}\nabla w_N \cdot \nabla v_N \dd \bm x, \quad \forall w_N, v_N \in X_N,
\end{equation}
and $P_N:L^2(\Omega) \to X_N$ is the ${L}^2$-orthogonal projection operator and fulfills the property
\begin{equation}
\label{P_N_property}
(P_Nv-v,w)=0, \quad \forall \ w \in X_N.
\end{equation}
It's easy to show, using the similar arguments in Theorem 2.1 in \cite{ShenTang2011} that $P_N$ is stable with respect to $L^2$-norm and $H^1$-norm, i.e., $\|P_Nu\|_0\lesssim \|u\|_0$ and $\|P_Nu\|_1\lesssim \|u\|_1$ for any $u \in H_p^1(\Omega)$. For convenience, we map the interval $[a_i,b_i]$ to $[0,2\pi]$ through the coordinate transform in each direction:
\begin{equation*}
y_i=\frac{x_i-a_i}{b_i-a_i}2\pi, \quad i = 1, \cdots, d.
\end{equation*}

Denote by $\widehat u_N$ the vector of the expansion coefficients of \eqref{eigenvalues}. Let $\widetilde k_i = \frac{2\pi k_i}{b_i-a_i}, i=1, \cdots, d$ and $\widetilde{\bm k}=(\widetilde{k_1}, \cdots, \widetilde{k_d})$, where $k_i$ is the element of $\bm k$. Due to the orthogonality property of basis functions in $X_N$, for any $\bm k \in \Lambda_N$, we can derive a set of ODEs for $\widehat{u}_N$,
\begin{equation}
\label{semi-discretization}
\left\{
\begin{split}
&(\widehat{u}_{N,\bm k})_t+|\widetilde{\bm k}|^2\widehat{u}_{N,\bm k}=\Big(\widehat P_N\widehat f(t, u_N,\nabla u_N)\Big)_{\bm k}, \quad \bm x \in \Omega, \\
&\widehat u_{N, \bm k}(0)=\Big(\widehat P_N \widehat {u_0}\Big)_{\bm k}, \quad \bm x \in \Omega,
\end{split}
\right.
\end{equation}
where $\widehat P_N$ is the truncation operator on the frequency space $X_N$ and satisfies that $\widehat P_N\widehat u = \{\widehat u \}_{\bm k \in \Lambda_N}$. For convience in the forthcoming analysis, we abbreviate $\widetilde{\bm k}$ as $\bm k$. It's obvious that for any $u\in H_p^1(\Omega)$, $\|\widehat P_N\widehat u\|_{\ell^2}\lesssim \|\widehat u\|_{\ell^2}$, where $\|\cdot\|_{\ell^2}$ is the $\ell^2$-norm for the tensor, i.e., for any tensor $\bm a=\{a_{\bm k}\}_{\bm k \in \Lambda_N}$, we have
\begin{equation*}
\|\bm a\|_{\ell^2}^2=\frac{1}{|\Lambda_N|}\sum\limits_{\bm k \in \Lambda_N}|a_{\bm k}|^2.
\end{equation*}

\subsection{Semilinear explicit exponential integrator in time}

The time interval $[0,T]$ is divided into $N_T>0$ subintervals $[t_n,t_{n+1}], n=0, \cdots, N_T-1$ with $\tau_n=t_{n+1}-t_n>0$ being the time step size at $t_n$. Denote $\{e^{-tL_N} \}_{t\geq 0}$ as the semigroup on $X_N$ with the infinitesimal generator $(-L_N)$. For brevity, we define $u(t):= u(t,\cdot),u_N(t):=u_N(t,\cdot)$ and $g(t,u(t)):=f(t,u(t),\nabla u(t))$. By the Duhamel formula, the semi-discrete solution $\widehat{u}_{N,\bm k}$ of the problem \eqref{semi-discretization} can be equivalently expressed as
\begin{equation}
\label{Duhamel}
\begin{split}
\widehat u_{N, \bm k}(t_{n+1}) =& e^{-\tau_n|\bm k|^2}\widehat u_{N,\bm k}(t_n) + \int_0^{\tau_n} e^{-(\tau_n-\sigma)|\bm k|^2}\Big(\widehat P_N\widehat g(t_n+\sigma,u_N(t_n+\sigma))\Big)_{\bm k}\dd \sigma,
\end{split}
\end{equation}

It's worthy to note that the first term at the right side of the formula \eqref{Duhamel} implies an operator $\widehat L_N$ in the frequency space, which is exactly the multiplier $\{|\bm k|^2 \}_{\bm k\in \Lambda_N}$, that is
\begin{equation*}
(\widehat L_N\widehat u_N)_{\bm k}=|\bm k|^2\widehat u_{N,\bm k},\quad \forall \ \bm k \in \Lambda_N.
\end{equation*}

Denote $L_N:X_N\to X_N$ as the induction operator of $\widehat L_N$. Then due to the definition of bilinear operator $a(\cdot,\cdot)$ in \eqref{bilinear}, we can derive that
\begin{equation}
\label{L_N}
a(w_N,v_N)=(L_Nw_N,v_N),\quad \forall \ w_N,v_N\in X_N.
\end{equation}

By the Rayleigh representation theorem and Parseval identity, we can easily derive that the eigenvalues of $L_N$ satisfy that
\begin{equation}
\label{eigenvalue2}
0\leq \lambda(L_N)\lesssim \Big(\frac N 2 \Big)^2.
\end{equation}

Denote $\widehat u_N^n$ as the fully-discrete numerical solution at time $t_n$ and $u_N^n$ as the inverse Fourier transformation result of $\widehat u_N^n$, i.e., $u_N^n=\text{ifft}(\widehat u_N^n)$. Applying the consistent explicit exponential Runge-Kutta method \cite{HochbruckOstermann2010} to approximating the integral of semi-discrete solution, then we propose the following EIFG method, that is for $n=0, \cdots, N_T-1$, $i=1,\cdots, s$,
\begin{subequations}
	\label{explicit_ETD1}
	\begin{align}
	\widehat u_{N,\bm k}^{n+1} &= e^{-\tau_n |\bm k|^2}\widehat u_{N,\bm k}^n + \tau_n \textstyle\sum\limits_{i=1}^s b_i(-\tau_n|\bm k|^2)(\widehat G_{ni})_{\bm k}, \label{nonlinear1}\\
	(\widehat U_{n,i})_{\bm k} &= e^{-c_i\tau_n|\bm k|^2}\widehat u_{N,\bm k}^n + \tau_n \textstyle\sum\limits_{j=1}^{i-1} a_{ij}(-\tau_n|\bm k|^2)(\widehat G_{nj})_{\bm k}, \label{intermediate11}\\
	(\widehat G_{nj})_{\bm k} &=\Big(\widehat P_N\widehat g(t_n+c_j\tau_n, U_{n,j})\Big)_{\bm k}, \label{intermediate12}
	\end{align}
\end{subequations}
where $U_{ni}=\text{ifft}(\widehat U_{ni})$, the interpolation nodes $c_1, \cdots, c_s$ are $s$ different nodes selected in $[0, 1]$ and the weights are denoted as:
\begin{equation}
\label{eq3-12}\left\{
\begin{split}
b_i(-\tau_n|\bm k|^2) &= \int_0^1 e^{-\tau_n(1-\theta)|\bm k|^2}l_i(\theta)\dd \theta,\quad i=1, \cdots, s\\
a_{ij}(-\tau_n|\bm k|^2) &= \frac{1}{\tau_n}\int_0^{c_i\tau_n}e^{-(c_i\tau_n-\tau)|\bm k|^2}l_j(\tau)\dd \tau, \quad i=1, \cdots, s, j=1, \cdots, i-1,
\end{split}
\right.
\end{equation}
where $\{l_i(\theta)\}_{i=1}^s$ are the Lagrange interpolation polynomials
\begin{equation*}
l_i(\theta) := \prod\limits_{m=1,m\neq i}^s \frac{\theta - c_m}{c_i-c_m}, \quad i = 1, \cdots s.
\end{equation*}

The above fully-discrete scheme \eqref{explicit_ETD1} is referred as the exponential integrator Fourier Galerkin (EIFG) method. Moreover, the first-order Euler exponential scheme (i.e., the number of RK stages $s=1$) is formulated as
\begin{equation}
\label{EIFE1}
\widehat u_{N,\bm k}^{n+1} = e^{-\tau_n|\bm k|^2}\widehat u_{N,\bm k}^n+\tau_n\varphi_1(-\tau_n|\bm k|^2)\Big(\widehat P_N\widehat g(t_n,u_N^n)\Big)_{\bm k},
\end{equation}
which is abbreviated as EIFG1. If the number of RK stages $s=2$, then the two interpolation nodes are taken as $c_1=0$ and $c_2\in (0,1]$ and the two-stage second-order exponential Runge-Kutta scheme is expressed as
\begin{equation}
\label{EIFE2}
\begin{split}
\widehat u_{N,\bm k}^{n+\frac 1 2} &= e^{-\tau_n|\bm k|^2}\widehat u_{N,\bm k}^n+c_2\tau_n\varphi_1(-c_2\tau_n|\bm k|^2)\Big(\widehat P_N\widehat g(t_n,u_N^n) \Big)_{\bm k},\\
\widehat u_{N,\bm k}^{n+1}&= e^{-\tau_n|\bm k|^2}\widehat u_{N,\bm k}^n +\tau_n\Big(\Big(\varphi_1(-\tau_n|\bm k|^2)-\frac{1}{c_2}\varphi_2(-\tau_n|\bm k|^2)\Big)\Big(\widehat P_N\widehat g(t_n,u_N^n)\Big)_{\bm k}\\
&+\frac{1}{c_2}\varphi_2(-\tau_n|\bm k|^2)\Big(\widehat P_N\widehat g(t_n+c_2\tau_n,u_{N,\bm k}^{n+\frac 1 2})\Big)_{\bm k}\Big),
\end{split}
\end{equation}
which is abbreviated as EIFG2.

Fully-discrete error analysis of \eqref{EIFE1} and \eqref{EIFE2} will be carefully studied in Section \ref{theory}. For higher order ($\geq 3$) explicit exponential Runge-Kutta schemes, more complicated order conditions are needed and interested reader can refer to \cite{HochbruckOstermann2005b} for details.

\section{Convergence analysis for exponential integrator Fourier Galerkin method}\label{theory}

In order to illustrate the convergence of EIFG method, we first reformulate \eqref{explicit_ETD1} in the operator form:
\begin{subequations}
	\label{explicit_ETD}
	\begin{align}
	\widehat u_{N}^{n+1} &= e^{-\tau_n \widehat L_N}\widehat u_{N}^n + \tau_n \textstyle\sum\limits_{i=1}^s b_i(-\tau_n\widehat L_N)\widehat G_{ni}, \label{nonlinear}\\
	\widehat U_{n,i} &= e^{-c_i\tau_n\widehat L_N}\widehat u_{N}^n + \tau_n \textstyle\sum\limits_{j=1}^{i-1} a_{ij}(-\tau_n\widehat L_N)\widehat G_{nj}, \quad i=1, \cdots, s\label{intermediate1}\\
	\widehat G_{nj} &=\widehat P_N\widehat g(t_n+c_j\tau_n, U_{n,j}), \quad i=1, \cdots. s.\label{intermediate2}
	\end{align}
\end{subequations}
According to the definition of $\widehat L_N$, we can derive that the eigenvalues of $\widehat L_N$ satisfy
\begin{equation}
\label{eigenvalue}
0 \leq \lambda(\widehat L_N)\lesssim \Big(\frac N 2\Big)^2.
\end{equation}

We will first give some preparation assumptions and lemmas for later analysis, which have been derived in \cite{HuangJu2023}. Then by following the similar arguments of \cite{HochbruckOstermann2005b,ThomeeVidar2006,LiuSun2016}, we will derive the semi-discrete error and fully-discrete error respectively for semilinear parabolic equation \eqref{eq3-1} equipped with periodic boundary condition. Moreover, we always assume $d \leq 3$ from now on.

In addition, since the operator $L_N$ is symmetric and positive definite, according to the Parseval indentity, we can establish an important relation between norms $\|\cdot\|_m$ (m=0, 1, 2) and the norm $\|\cdot\|_{\ell^2}$,
\begin{equation}
\label{equal_norm}
\|v\|_1 \eqsim \|L_N^{\frac 1 2}v\|_0\eqsim \|\widehat L_N^{\frac 1 2}\widehat v\|_{\ell^2}, \quad \|v\|_2 \eqsim \|L_Nv\|_0\eqsim \|\widehat L_N\widehat v\|_{\ell^2}, \quad \forall v \in X_N.
\end{equation}

\subsection{Some preliminary lemmas}

We first present some fundamental estimates for the semigroup $\left\{e^{-t L_N} \right\}_{t \geq 0}$, which are important for the error analysis of EIFG method and will be used frequently later on. 

\begin{lemma}\cite{HuangJu2023}
	\label{lemma2}
	\begin{itemize}
		\item[\rm (i)] For any given parameter $\gamma\geq 0$, it holds
		\begin{equation}\label{lemma2-1}
		\|e^{-t L_N}\|_0+\|t^{\gamma}L_N^{\gamma}e^{-t L_N}\|_0 \lesssim 1, \quad \forall\  t>0,\forall\  N>0.
		\end{equation}
		\item[\rm (ii)] For any given parameter $0 \leq \gamma \leq 1$, it holds
		\begin{equation}
		\label{lemma2-2}
		\Big\|tL_N^{\gamma}\textstyle\sum\limits_{j=1}^{n-1}e^{-jtL_N}\Big\|_0 \lesssim 1, \quad \forall\  t>0, \forall\  N > 0.
		\end{equation}
		\item[\rm (iii)] For any given parameter $0 \leq \gamma \leq 1$, it holds
		\begin{equation}\label{lemma2-3}
		\|\phi(-t L_N)\|_0+\|t^{\gamma}L_N^{\gamma}\phi(-t L_N)\|_0 \lesssim 1, \quad \forall\  t>0, \forall\  N>0,
		\end{equation}
		where $\phi(-t L_N)=b_i(-t L_N)$ or $\phi(-t L_N)=a_{ij}(-t L_N), i, j=1, \cdots, s$.
		
	\end{itemize}
\end{lemma}

\begin{proof}
	Due to the orthogonality property of basis functions in $X_N$, we can easily know $L_N$ is a linear symmetric operator on $X_N$ and its eigenvalues satisfy condition \eqref{eigenvalue}. Therefore, we obtain
	\begin{equation}
	\label{semigroup-1}
	\|e^{-\tau L_N }\|_{0}\lesssim 1.
	\end{equation}
	Similar as the derivation of Lemma 3.1 in \cite{HuangJu2023}, $L_N$ is a symmetric operator on $X_N$, so is $\tau^{\gamma}L_N^{\gamma}e^{-\tau L_N}$. Thus it holds
	\begin{equation}
	\label{semigroup-2}
	\|\tau ^{\gamma}L_N^{\gamma}e^{-\tau L_N}\|_0=\max_{\lambda \in \lambda(L_N)} |(\tau \lambda)^{\gamma}e^{-\tau\lambda}|.
	\end{equation}	
	Let us consider an auxiliary function
	$
	g(x)=x^{\gamma}e^{-x}$ for $x\ge 0$.
	The derivative with respect to $x$ is $g^{\prime}(x)=x^{\gamma-1}e^{-x}(\gamma-x)$, so the maximum of $g(x)$ is taken at $x=\gamma$, which implies
	\begin{equation}\label{ttt}
	g(x)\le \gamma^{\gamma}e^{-\gamma},\quad  \forall\,x\ge 0.
	\end{equation}
	Combination of \eqref{ttt} with \eqref{semigroup-2}  immediately gives us
	$
	\|\tau ^{\gamma}L_N^{\gamma}e^{-\tau L_N}\|_0\le \gamma^{\gamma}e^{-\gamma},
	$
	which together with  \eqref{semigroup-1} then directly deduces  \eqref{lemma2-1}.
	
	Following the similar arguments of Lemma 3.1 in \cite{HuangJu2023}, \eqref{lemma2-2} and \eqref{lemma2-3} can be easily derived by using the eigenvalue estimate \eqref{eigenvalue} of $L_N$. Therefore, we will omit the details of provement.
	
\end{proof}

Similar to \cite{ThomeeVidar2006}, we introduce the mild growth condition for the function $f$ and some regularity conditions required for the exact solution $u(t)$ in order to carry out convergence and error analysis of EIFG method.

\begin{assumption}
	\label{assumption4}
	The function $f(t,\xi,\eta)$ and $D^{\bm{\alpha}} f(t,\xi,\eta)$ grow mildly with respect to $\xi$ and $\eta$ for any $|\bm{\alpha}|=1$, i.e., there exists a number $p>0$ for $d=1,2$ or $p \in (0,2]$ for $d=3$ such that
	\begin{subequations}
		\begin{align}
		\label{mild_growth1} \Big|\frac{\partial f}{\partial \xi}(t,\xi,\eta) \Big| \lesssim 1+|\xi|^p, &\quad \Big|\frac{\partial f}{\partial \eta}(t,\xi,\eta)\Big|\lesssim 1+|\eta|^p,\quad \forall t, \xi, \eta \in \mathbb{R}.\\
		\label{mild_growth2} \Big|\frac{\partial D^{\bm \alpha} f}{\partial \xi}(t,\xi,\eta) \Big|\lesssim 1+|\xi|^p, &\quad \Big|\frac{\partial  D^{\bm \alpha} f}{\partial \eta}(t,\xi,\eta) \Big|\lesssim 1+|\eta|^p, \quad \forall t, \xi, \eta \in \mathbb{R}.
		\end{align}
	\end{subequations}
	
\end{assumption}

\begin{assumption}
	\label{assumption5}
	The function $f(t,\zeta,\nabla \zeta)$ is sufficiently smooth with respect to $t,\xi,\eta$, i.e., for any given constant $K_1, K_2>0$, it holds
	\begin{equation}
	\label{smooth_function}
	\textstyle\sum\limits_{|\bm {\alpha}| \leq 2}\Big|D^{\bm{\alpha}}f(t,\xi,\eta) \Big|\lesssim 1, \quad \forall t \in [0, T], \xi \in [-K_1, K_1], \eta \in [-K_2, K_2].
	\end{equation}
\end{assumption}

\begin{assumption}
	\label{assumption6}
	The exact solution $u(t)$ satisfies some of the following regularity conditions
	\begin{subequations}
		\begin{align}
		\label{regularity1} \sup\limits_{0\leq t \leq T}\|u(t)\|_{m,\Omega} &\lesssim 1,\\
		\label{regularity2} \sup\limits_{0\leq t \leq T}\|u_t(t)\|_{0,\infty, \Omega} &\lesssim 1,\\
		\label{regularity3} \sup\limits_{0\leq t \leq T}\|u_{tt}(t)\|_{0,\infty, \Omega} &\lesssim 1,
		\end{align}
	\end{subequations}
	where the hidden constants may depend on $T$ and the parameter $m$ has been defined in \eqref{eq3-1}.
\end{assumption}

Now, we will propose the relation between the function $f(t,u)$ and $\widehat{f}(t,u,\nabla u)$, which is very important for the forthcoming analysis.

\begin{lemma}
	\label{locally_Lipschitz}
	Suppose that the function $f$ satisfies Assumption \ref{assumption4}, and the exact solution $u(t)$ fulfills \eqref{regularity1} in Assumption \ref{assumption6} with $m \geq 2$. Then $\widehat f$ and $f$ are locally-Lipschitz continuous in a strip along the exact solution $u(t)$, i.e., for any given constant $R>0$,
	\begin{subequations}
	\begin{align}
		\|\widehat f(t,v,\nabla v)-\widehat f(t,w,\nabla w)\|_0 &\lesssim \|v-w\|_2,\label{locally_Lipschitz_eq}\\
		\|f(t,v,\nabla v) - f(t,w,\nabla w)\|_1 &\lesssim \|v-w\|_2,\label{locally_Lipschitz_eq2}
	\end{align}
	\end{subequations}
	for any $t \in [0, T]$ and $v, w\in H^m_p(\Omega)$ satisfying
	\begin{equation*}
	\label{neighbor}
		\max\{\|v-u(t)\|_2, \|w-u(t)\|_2 \}\leq R,
	\end{equation*}
	where the hidden constant in \eqref{locally_Lipschitz_eq} may depend on $R$.
\end{lemma}

\begin{proof}
	As for \eqref{locally_Lipschitz_eq}, recalling the derivation of Lemma 3.2 in \cite{HuangJu2023}, we can easily derive the above conclusion with minor modifications if the inequality
	\begin{equation}
	\label{locally_Lipschitz_mid1}
	\begin{split}
	\|\widehat f(t,v,\nabla v)-\widehat f(t,w,\nabla w)\|_0^2 \lesssim& \int_{\Omega}(1+|v|^p)^2|v- w|^2 \dd \bm x + \int_{\Omega} (1+|w|^p)^2|v-w|^2 \dd \bm x\\
	&+ \int_{\Omega}(1+|\nabla v|^p)^2|\nabla v- \nabla w|^2 \dd \bm x + \int_{\Omega} (1+|\nabla w|^p)^2|\nabla v-\nabla w|^2 \dd \bm x
	\end{split}
	\end{equation}
	is held. Now, we will prove the estimate \eqref{locally_Lipschitz_mid1}.
	
	Since the function $f$ satisfies Assumption \ref{assumption4}, we can obtain by the Lagrange mean value theorem,
	\begin{equation*}
	\begin{split}
	&\|\widehat f(t,v,\nabla v)-\widehat f(t,w,\nabla w)\|_{\ell^2}^2\\
	=& \Big(\prod_{i=1}^d N_i\Big)^{-1}\sum\limits_{\bm k\in \Lambda_N}\Big|\Big(\widehat f(t,v,\nabla v)-\widehat f(t,w,\nabla w) \Big)_{\bm k}\Big|^2\\
	=&\Big(\prod_{i=1}^d N_i\Big)^{-1} \sum\limits_{\bm k\in \Lambda_N}\Big|\Big(\frac{\partial f}{\partial \xi}(t,\zeta,\nabla \zeta)(v-w)+\frac{\partial f}{\partial \eta}(t,\zeta,\nabla \zeta)(\nabla v-\nabla w),e^{i\bm k\cdot\bm x} \Big)\Big|^2\\
	=& \Big(\prod_{i=1}^d N_i\Big)^{-1} \sum\limits_{\bm k\in \Lambda_N}\Big|\Big(\frac {1}{2\pi}\int_{\Omega} \frac{\partial f}{\partial \xi}(t,\zeta, \nabla \zeta)(v-w)e^{-i\bm k\cdot\bm x}\dd \bm{x}  +\frac{1}{2\pi}\int_{\Omega}\frac{\partial f}{\partial \eta}(t,\zeta,\nabla\zeta)(\nabla v-\nabla w)e^{-i\bm k\cdot \bm x}\dd \bm x \Big)\Big|^2\\
	\lesssim& \Big(\prod_{i=1}^d N_i\Big)^{-1} \sum\limits_{ {\bm k}\in \Lambda_N} \Big(\int_{\Omega}\Big|\frac{\partial f}{\partial \xi}(t,\zeta,\nabla \zeta) \Big|^2|v-w|^2\dd \bm x+\int_{\Omega} \Big|\frac{\partial f}{\partial \eta}(t,\zeta,\nabla\zeta) \Big|^2|\nabla v-\nabla w|^2\dd \bm x \Big)\\
	\lesssim& \Big(\prod_{i=1}^d N_i\Big)^{-1} \Big(\int_{\Omega}(1+|\zeta|^p)^2|v-w|^2\dd \bm x+\int_{\Omega} (1+|\nabla \zeta|^p)^2|\nabla v-\nabla w|^2 \dd \bm x\Big)\\
	\lesssim& \int_{\Omega} (1+|\zeta|^p)^2 |v-w|^2\dd \bm x+\int_{\Omega} (1+|\nabla \zeta|^p)^2|\nabla v-\nabla w|^2\dd \bm x,
	\end{split}
	\end{equation*}
	where $\zeta(\bm x)=\theta(\bm x)v(\bm x)+(1-\theta(\bm x))w(\bm x)$ for some $\theta(\bm x)\in [0,1]$. Therefore, it's clear that
	\begin{equation*}
	\begin{split}
	\|\widehat f(t,v,\nabla v)-\widehat f(t,w,\nabla w)\|_{\ell^2}^2 \lesssim& \int_{\Omega} \Big(1+|v|^p \Big)^2| v- w|^2 \dd \bm x + \int_{\Omega} \Big(1+|w|^p \Big)^2|v-w|^2 \dd \bm x\\
	&+ \int_{\Omega} \Big(1+|\nabla v^p \Big)^2|\nabla v-\nabla w|^2 \dd \bm x + \int_{\Omega} \Big(1+|\nabla w|^p \Big)^2 |\nabla v- \nabla w|^2 \dd \bm x,
	\end{split}	
	\end{equation*}
	which implies that the inequality \eqref{locally_Lipschitz_mid1} is held. Then the locally-Lipschitz continuity of $\widehat{f}$ can be easily obtained by following the similar arguments of Lemma 3.2 in \cite{HuangJu2023}. In the same way, \eqref{locally_Lipschitz_eq2} can also be obtained.

\end{proof}

\subsection{Semi-discrete error estimate}

In this subsection, we will derive the semi-discrete error estimate of EIFG method. Notations are same as that have been defined in Section \ref{algorithm-description}. According to the definition in \eqref{P_N_property}, the $L^2$-orthogonal projector can be further reformulated as: for any $u \in L^2(\Omega)$,
\begin{equation*}
P_Nu=\sum\limits_{ {\bm k}\in \Lambda_N} \widehat u_{\bm k}\phi_{{\bm k}}.
\end{equation*}

The following lemmas have readily come from \cite{CanutoHussaini1988,CanutoHussaini2006}.
\begin{lemma}
	\label{lemma3}
	For $m \geq 2$ and $ v \in H^m_p(\Omega)$, we have
	\begin{equation}
	\label{lemma3_eq1}
	\|v-P_Nv\|_2 \lesssim N^{2-m}\|v\|_m.
	\end{equation}
\end{lemma}

Define the elliptic projection $\Pi_N:H^1_p(\Omega) \to X_N$ such that
\begin{equation}
\label{elliptic_projection}
a(\Pi_N u - u,v)=0, \quad \forall \ v \in X_N.
\end{equation}
Then we have the following continuity and coercivity.

\begin{lemma}
	For the bilinear form defined in \eqref{bilinear}, we can derive that
	\label{lemma4}
	\begin{subequations}
		\label{lemma4_eq1}
		\begin{align}
		a(u,v) &\leq |u|_1|v|_1, \quad &&\forall \ u,v \in H^1_p(\Omega),\label{lemma4_eq1-1}\\
		a(v,v) &= |v|_1^2, \quad &&\forall \ v \in H_p^1(\Omega) \label{lemma4-eq1-2}.
		\end{align}
	\end{subequations}
\end{lemma}

Denote $\bm j=(j_1,\cdots, j_d),0\leq j_i\leq N_i,i=1,\cdots,d$ and $\bm{x_j}=(x_1^{j_1},\cdots,x_d^{j_d})^T$. Define $I_N:H_p^1(\Omega)\to X_N$ as the interpolation operator on $X_N$ such that
\begin{equation*}
I_Nu=\sum\limits_{ {\bm k}\in \Lambda_N}\widetilde{u}_{\bm k}\phi_{ {\bm k}},\quad \widetilde{u}_{\bm k}=\Big( \prod\limits_{i=1}^d N_i\Big)^{-1}\sum\limits_{ {\bm j}\in \Lambda_N}u(\bm{x_j})\overline{\phi_{\bm k}(\bm {x_j})}.
\end{equation*}

The following lemmas have readily come from \cite{Guo1998,CanutoHussaini1988}.

\begin{theorem}
	\label{interpolation-error}
	The exact solution $u(t)$ fulfills \eqref{regularity1} in Assumption \ref{assumption6} and $m>\frac{d}{2}$, then it holds
	\begin{equation*}
	\|I_Nu-u\|_{l}\lesssim N^{l-m}|u|_m,\quad 0\leq l\leq m.
	\end{equation*}	
\end{theorem}

\begin{theorem}[Inverse inequality for Fourier space]
	\label{inverse_inequality}
	For any $u\in X_N$ and $1\leq p \leq \infty$,
	\begin{equation*}
	\|u\|_m \lesssim N^{m-k}\|u\|_{k},\quad m\geq 1,\ 0\leq k \leq m.
	\end{equation*}
\end{theorem}

Until now, we have proposed some preliminary lemmas for semi-discrete error estimate. Define $\rho:=u-\Pi_Nu,\theta:=\Pi_Nu-u_N$. Later, we will give the semi-discrete error estimate for EIFG method. Rewritting the semi-discrete error $ u- u_N$ as a sum of two following terms
\begin{equation}
\label{semi-discrete-split}
u - u_N = (u - I_N u)+(I_N u-u_N).
\end{equation}
The following theorem shows the standard estimate for $\rho$ and $\rho_t$, which has been widely presented in many literatures, such as \cite{ThomeeVidar2006,LiuSun2016}.
\begin{theorem}
	\label{semi-discrete-error1}
	Assume the exact solution $u(t)$ fulfills \eqref{regularity1} in Assumption \ref{assumption6}. Then the following estimates are held.
	\begin{subequations}
		\begin{align}
		\|\rho(t)\|_0 + N^{-1}\|\rho(t)\|_1 &\lesssim N^{-m}\sup\limits_{0 \leq \eta \leq T}\| u(\eta)\|_m, \quad \forall \ t \in [0, T],\label{semi-discrete-error-eq1}\\
		\|\rho_t(t)\|_0+N^{-1}\|\rho_t(t)\|_1 & \lesssim N^{-m}\sup\limits_{0 \leq \eta \leq T}\|u_t(\eta)\|_m, \quad \forall \ t \in [0, T]. \label{semi-discrete-error-eq2}
		\end{align}
	\end{subequations}
\end{theorem}

Following the similar arguments of Theorem 14.2 in \cite{ThomeeVidar2006}, we derive the semi-discrete error for EIFG method.
\begin{theorem}
	\label{semi-discrete-error2}
	Suppose that the function $f$ satisfies Assumptions \ref{assumption4} and \ref{assumption5}, and the exact solution fulfills \eqref{regularity1} in Assumption \ref{assumption6}. There exists a constant $N_0>0$ such that if the number of spatial mesh $N\geq N_0$, then
	\begin{equation}
	\label{semi-discrete-error-eq}
	\|u(t)-u_N(t)\|_0+N^{-1}\|u(t)-u_N(t)\|_1 \lesssim N^{-m}, \quad \forall \ t \in [0, T],
	\end{equation}
	where the hidden constant is independent of $N$.
\end{theorem}

\begin{proof}
	Applying $\chi \in X_N$ to \eqref{eq3-11}, we can derive that
	\begin{align}
	\label{eq5-1}
	&(\theta_t,  \chi)+(\nabla \theta,  \nabla \chi)\nonumber\\
	=& (u_{N, t},  \chi)-({\Pi}_{N}u_t, \chi) + (\nabla u_N,  \nabla\chi)-(\nabla {\Pi}_{N}u, \nabla\chi)\nonumber\\
	=&(f(t,  u_N, \nabla u_N), \chi)-({\Pi}_{N}u_t,  \chi)-(\nabla u, \nabla \chi)\nonumber\\
	=&(f(t, u_N, \nabla u_N)-f(t, u, \nabla u), \chi) - (\rho_t,  \chi).
	\end{align}
	
	Then we will derive
	\begin{equation}
	\label{eq5-2}
	\Big| (f(t,u_N,\nabla u_N)-f(t,u,\nabla u),\theta)\Big| \lesssim \|u_N-u\|_0\|\nabla \theta\|_0.
	\end{equation}
	
	($\rm i$) When $d\leq 2$, due to the norm equivalence, we have
	\begin{equation}
	\label{eq5-3}\|\theta\|_{0,q,\Omega}\lesssim \|\theta\|_0,\quad \forall \ 2<q<\infty.
	\end{equation}
	According to H\"{o}lder inequality and \eqref{eq5-3}, taking $q^{-1}+(q')^{-1}=1$, then
	\begin{align*}
	&\Big|(f(t, u_N, \nabla u_N)-f(t, u, \nabla u), \theta) \Big|\\
	\lesssim& \|f(t, u_N, \nabla u_N)- f(t, u,  \nabla u)\|_{0, q', \Omega}\|\theta\|_{0, q, \Omega}\\
	\lesssim& \|f(t, u_N, \nabla u_N)-f(t, u, \nabla u)\|_{0, q', \Omega}\|\nabla \theta\|_0.
	\end{align*}
	
	Applying H\"{o}lder inequality again, since the function $f$ satisfies the Assumption \ref{assumption4}, we can derive that
	\begin{align*}
	&\|f(t, u_N, \nabla u_N)- f(t, u,  \nabla u)\|_{0, q',  \Omega}^{q'}\\
	\lesssim& \int_{\Omega} |u_N-u|^{q'}(1+| u|+| u_N|+|\nabla u|+|\nabla u_N|)^{pq'} \dd \bm x\\
	\lesssim& \Big(\int_{\Omega}|u_N-u|^2\dd \bm x \Big)^{\frac{q'}{2}}\Big(\int_{\Omega}(1+| u|+|u_N|+|\nabla u|+|\nabla u_N|)^{pq'r}\dd \bm x \Big)^{\frac 1 r}\\
	\lesssim& \|u_N- u\|_0^{q'}\Big(1+\|u_N\|_{0, r, \Omega}+\|u\|_{0, r, \Omega}+\|\nabla u_N\|_{0, r, \Omega}+\|\nabla u\|_{0, r, \Omega} \Big)^{pq'},
	\end{align*}
	where $r=\frac{2pq}{q-2}$. For a proper $q>0$, we have $2<r<\infty$, then $\| u_N\|_{0, r, \Omega}\lesssim \|\nabla u_N\|_0$. Furthermore, since $u(t)$ fulfills \eqref{regularity1} in the Assumption \ref{assumption6}, due to the norm equivalence on $X_N$, we can derive that
	\begin{align*}
	\|f(t, u_N, \nabla u_N)-f(t, u, \nabla u)\|_{0, q', \Omega}\lesssim \|u_N-u\|_0\Big(1+\|\nabla u_N\|_0 \Big)^p.
	\end{align*}
	
	Later, we will prove that when $N$ is sufficiently large, $\|\nabla u_N\|_0$ is bounded. Inserting $\chi=2\theta_t$ into \eqref{eq5-1}, we can derive that
	\begin{align}
	\label{eq5-4}
	(\theta_t, 2\theta_t)+a(\theta, 2\theta_t)=&\Big(f(t, u_N, \nabla u_N)-f(t, u, \nabla u), 2\theta_t \Big)\notag\\
	&\quad-(\rho_t, 2\theta_t).
	\end{align}
	
	Then \eqref{eq5-4} can be reformulated as
	\begin{align}
	\label{eq5-5}
	\frac{\dd }{\dd t}\|\nabla \theta\|_0^2&=(f(t, u_N, \nabla u_N)- f(t, u, \nabla u), 2\theta_t)-(\theta_t+\rho_t, 2\theta_t)\nonumber\\
	&\leq \|f(t, u_N, \nabla u_N)-f(t, u, \nabla u)\|_0^2+\|\rho_t\|_0^2.
	\end{align}
	
	By using the estimation of $\rho$ in \eqref{semi-discrete-error-eq1}, since the function $f$ satisfies the Assumption \ref{assumption4}, \eqref{eq5-5} can be rewritten as 
	\begin{align}
	\label{eq5-6}
	\frac{\dd}{\dd t}\|\nabla \theta\|_0^2 \lesssim& \|f(t, u_N, \nabla u_N)-f(t, \Pi_{N}u, \nabla (\Pi_{N}u))\|^2+N^{2-2m}\nonumber\\
	& \quad+ \|f(t,  \Pi_{N}u, \nabla (\Pi_{N}u))-f(t, u, \nabla u)\|^2.
	\end{align}
	
	As for the estimation of $\|f(t, \Pi_{N}u, \nabla (\Pi_{N}u))-f(t, u, \nabla u)\|_0$, since the function $f$ is mild growth (can also see the Assumption \ref{assumption4}) and the estimation of $\|\nabla \rho\|_0$ in \eqref{semi-discrete-error-eq1}, we can obtain that
	\begin{align}
	\label{eq5-7}
	&\|f(t, \Pi_{N}u, \nabla (\Pi_{N}u))-f(t, u, \nabla u)\|_0^2\notag\\
	\lesssim& \int_{\Omega} \Big(1+|\Pi_{N} u|+|\nabla \Pi_{N}u| \Big)^{2p}|\Pi_{N} u- u|^2 \dd \bm x\nonumber\\
	\lesssim& \Big(\int_{\Omega} \rho^q\dd \bm x \Big)^{\frac 2 q}\Big(\int_{\Omega}(1+|\Pi_{N} u|+|\nabla(\Pi_{N}u)|)^r\dd \bm x \Big)^{\frac{2p}{r}}\nonumber\\
	=& \|\rho\|_{0, q, \Omega}^2\Big\|1+\Pi_{N} u+\nabla(\Pi_{N} u)\Big\|_{0, r, \Omega}^{2p}\nonumber\\
	\lesssim& \|\nabla\rho\|_0^2 \lesssim N^{2-2m},
	\end{align}
	where the last two inequality uses the fact that when $2<r < \infty$, $\|\rho\|_{0, q, \Omega}\lesssim \|\nabla \rho\|_0$, $\|\Pi_{N} u\|_{0, r, \Omega}\lesssim \|\nabla(\Pi_{N} u)\|_0\lesssim 1$ and $\|\nabla (\Pi_{N} u)\|_{0, r, \Omega}\lesssim \|\nabla(\Pi_{N} u)\|_0\lesssim 1$.
	
	As for the estimation of  $\|f(t, u_N, \nabla u_N)-f(t, \Pi_{N}u, \nabla \Pi_{N} u)\|_0$, following the similar arguments of \eqref{eq5-7}, we can derive that
	\begin{align}
	\label{eq5-8}
	&\|f(t, u_N, \nabla u_N)-f(t,  \Pi_{N}u, \nabla (\Pi_{N}u))\|_0^2\nonumber\\
	\lesssim& \|\nabla \theta\|_0^2\|1+ u_N+\nabla u_N\|_{0, r, \Omega}^{2p}\nonumber\\
	\lesssim& \|\nabla \theta\|_0^2\Big(1+\|\nabla u_N\|_0\Big)^{2p}.
	\end{align}
	
	According to the definition of elliptic projector $\Pi_{N}$ in \eqref{elliptic}, we can know that
	\begin{equation}
	\label{eq5-9}
	\|\nabla( \Pi_{N} u)\|_0\lesssim \|\nabla u\|_0 \lesssim 1.
	\end{equation}
	
	Inserting \eqref{eq5-9} into \eqref{eq5-8}, we can obtain that
	\begin{align}
	\label{4-thm3-mid10}
	&\|f(t, u_N, \nabla u_N)-f(t, \Pi_{N}u, \nabla(\Pi_{N}u))\|_0^2\nonumber\\
	\lesssim& \|\nabla \theta\|_0^2\Big( 1+\|\nabla \theta\|_0+\|\nabla (\Pi_{N} u\|_0)\Big)^{2p}\nonumber\\
	\lesssim& \|\nabla\theta\|_0^2\Big(1+\|\nabla \theta \|_0 \Big)^{2p}.
	\end{align}
	
	Let $\bar t_N\in[0, T]$ is sufficiently large such that $\|\nabla\theta \|_0\leq 1$ when $t\in[0, \bar t_N]$. Then, 
	\begin{equation*}
	\frac{\dd}{\dd t}\|\nabla \theta\|_0^2\lesssim \|\nabla \theta\|_0^2+N^{2-2m}.
	\end{equation*}
	
	Then, for constants $C_1, C_2$ independent of $\bar t_N$, we have
	\begin{equation*}
	\|\nabla \theta\|_0\leq C_1e^{C_2\bar t_N}N^{1-m} \lesssim N^{1-m},
	\end{equation*}
	
	Therefore, we have finished the error estimate for $\|\nabla \theta\|_0$ in \eqref{semi-discrete-error-eq}.
	
	In the same way, we can derive that, for the whole time interval $[0,T]$,
	\begin{equation*}
	\|\nabla \theta\|_0\leq 1,  \quad \forall \ N \geq N_0.
	\end{equation*}
	Therefore,
	\begin{equation*}
	\|\nabla u_N\|_0\leq \|\nabla (\Pi_{N} u)\|_0+\|\nabla \theta\|_0\lesssim 1,
	\end{equation*}
	which implies that $\|\nabla u_N\|_0$ is uniformly bounded on $[0, T]$. Thus, when $d=2$, \eqref{eq5-2} is held. 
	
	($\rm ii$) When $d=3$. Taking $q=6$, then for $r=6$, we have $\|\chi \|_{0, r, \Omega}\lesssim \|\nabla \chi \|_0$.  Since $p\leq 2$, we have $r=\frac{2pq}{q-2}=3p\leq q$, and the following analysis is similar to $d=2$. Thus, \eqref{eq5-2} is also held for $d=3$.

	Taking $\chi=\theta$, \eqref{eq5-1} can be rewritten as
	\begin{align*}
	\frac 1 2\frac{\dd}{\dd t}\|\theta \|_0^2 \lesssim \|\theta\|_0^2 + \|\rho\|_0^2+\|\rho_t\|_0^2
	\end{align*}
	Since \eqref{eq5-2} is held, then by using the estimation of $\|\rho\|_0,  \|\rho_t\|_0$ in \eqref{semi-discrete-error-eq1}-\eqref{semi-discrete-error-eq2}, we can derive that
	\begin{equation*}
	\|\theta\|_0 \lesssim N^{-m}.
	\end{equation*}
	
	Combining with the estimation of $\rho$, $\theta$ and triangle inequality, we can obtain that
	\begin{equation*}
	\begin{split}
	&\|u(t)-u_N(t)\|_0 + N^{-1}\|u(t)-u_N(t)\|_1\\
	\lesssim &\|u(t)-\Pi_{N} u\|_0+\|\Pi_{N} u-u_N(t)\|_0\\
	&\quad+N^{-1}\|u(t)-\Pi_{N} u\|_1+N^{-1}\|\Pi_{N} u-u_N\|_1\lesssim N^{-m}.
	\end{split}
	\end{equation*}
	
\end{proof}

\begin{theorem}
	\label{elliptic_error}
	Assume the exact solution $u(t)$ fulfills \eqref{regularity1} in Assumption \ref{assumption6} with $m \geq 2$. Then the following error estimate is held,
	\begin{equation*}
	\|u(t)-u_N(t)\|_2 \lesssim N^{2-m}, \quad \forall \ t \in [0, T],
	\end{equation*}
	where the hidden constant is independent of $N$.
\end{theorem}

\begin{proof}
	By using the triangle inequality,
	\begin{equation}
	\label{mid1}
	\|u(t)-u_N(t)\|_2 \leq \|u(t)-I_Nu(t)\|_2 + \|I_Nu(t)-u_N(t)\|_2,
	\end{equation}
	where $I_N$ is the interpolation operator on $X_N$.
	
	Recalling the Theorem \ref{interpolation-error}, we can derive the interpolation error as follows:
	\begin{equation}
	\label{mid2}
	\|u-I_Nu\|_2 \lesssim N^{2-m}\|u\|_m.
	\end{equation}
	
	With the help of the Theorem \ref{inverse_inequality}, Theorem \ref{semi-discrete-error1} and Lemma 4.3 in \cite{LiuSun2016}, we can derive that
	\begin{align}
	\label{mid3}
	\|I_Nu(t)-u_N(t)\|_2 &\lesssim N\|I_Nu(t)-u_N(t)\|_1 \nonumber\\
	&\leq N\|I_Nu(t) - u(t)\|_1 + N\|u(t)-u_N(t)\|_1\nonumber\\
	&\lesssim N^{2-m}.
	\end{align}
	
	Inserting \eqref{mid2} and \eqref{mid3} to \eqref{mid1}, we can easily derive the conclusion.

\end{proof}

\subsection{Fully-discrete error estimate}\label{EIFG-fully-discrete}

In this subsection, we will derive the fully-discrete error estimate of EIFG method. For simplicity, we will assume the time partition is uniform, i.e., $\Delta t = \tau_0 = \cdots = \tau_{N_T-1}$ and $t_n=n\Delta t$. Let $u_N(t)$ be the solution of the semi-discrete (in space) problem \eqref{semi-discretization} (or \eqref{Duhamel}), and $\{\widehat u_N^n \}$ be the fully-discrete solution produced by the EIFG method \eqref{explicit_ETD} and $u_N^n=\ifft(\widehat u_N^n) $. For the error between the exact solution $u(t)$ and the fully-discrete solution $\{u_N^n \}$ measured in the $H^2$-norm, we will derive the error into two parts. By the triangle inequality,
\begin{equation}
\label{fully-discrete-error}
\|u(t_n)-u_N^n\|_2 \leq \|u(t_n)-u_N(t_n)\|_2+\|u_N(t_n)-u_N^n\|_2,
\end{equation}
and the semi-discrete error has been proposed in Theorem \ref{semi-discrete-error2}. Then we will focus on the estimate of $\|u_N(t_n)-u_N^n\|_2$ by following the similar arguments of \cite{HochbruckOstermann2005a,HochbruckOstermann2005b,HuangJu2023}.

According to our previous work \cite{HuangJu2023}, it's crucial to remove the dependence of the hidden constants on $N$ (the number of spatial meshes) in estimating $\|u_N(t_n)-u_N^n\|_2$. In order to achieve this goal, we convert the semi-discrete solution $u_h(t_{n+1})$ ($n=0,1,\cdots,N_T-1$) into the sum of the following two parts for further analysis:
\begin{equation}
\label{eq1-6}
\begin{split}
\widehat u_N(t_{n+1}) =\;& e^{-\Delta t \widehat L_N }\widehat u_N(t_n)+\int_{0}^{\Delta t } e^{-(\Delta t  - \sigma)\widehat L_N}\widehat P_N\widehat g(t_n+\sigma, u_N(t_n+\sigma))\dd \sigma\\
=\;& e^{-\Delta t  \widehat L_N}\widehat u_N(t_n) + \int_0^{\Delta t } e^{-(\Delta t  - \sigma)\widehat L_N}\widehat P_N\widehat g(t_n+\sigma, u(t_n+\sigma)) \dd \sigma\\
&\hspace{-0.2cm}+\int_0^{\Delta t } e^{-(\Delta t  - \sigma)\widehat L_N}\Big(\widehat P_N\widehat g(t_n+\sigma, u_N(t_n+\sigma))-\widehat P_N\widehat g(t_n+\sigma, u(t_n+\sigma)) \Big) \dd \sigma.
\end{split}
\end{equation}

Define the following functions:
\begin{equation}
\left\{
\label{function}
\begin{split}
\psi_i(-\Delta t \widehat L_N) &= \varphi_i(-\Delta t \widehat L_N)-\textstyle\sum\limits_{k=1}^s b_k(-\Delta t \widehat L_N)\frac{c_k^{i-1}}{(i-1)!}, \quad i=1, \cdots, s,\\
\psi_{j, i}(-\Delta t \widehat L_N) &= \varphi_j(-c_i\Delta t \widehat L_N)c_i^j - \textstyle\sum\limits_{k=1}^{i-1} a_{ik}(-\Delta t \widehat L_N)\frac{c_k^{j-1}}{(j-1)!}, \quad i,j =1, \cdots, s.
\end{split}
\right.
\end{equation}
We also denote $g^{(k)}(t, u(t)) = \frac{\dd^k}{\dd t^k}g(t,u(t))=\frac{\dd^k}{\dd t^k}f(t,u(t),\nabla u(t))$ as the $k$-th full differentiation of $g$ with respect to $t$ and $\widehat g^{(k)}$ is the Fourier transformation result of $g^{(k)}$, i.e., $\widehat g^{(k)}=\text{fft}(g^{(k)})$. By comparing (\ref{eq1-6}) with the fully-discrete scheme \eqref{explicit_ETD}, we then obtain
\begin{eqnarray}
\widehat u_N(t_n+c_i\Delta t )&= &e^{-c_i \Delta t  \widehat L_N}\widehat u_N(t_n)+\Delta t  \textstyle\sum\limits_{j=1}^{i-1} a_{ij}(-\Delta t  \widehat L_N)\nonumber\\
&& \qquad \cdot \widehat P_N\widehat g(t_n+c_j \Delta t , u(t_n+c_j\Delta t )) + \widehat \delta_{ni},\label{111}\\
\widehat u_N(t_{n+1}) &= &e^{-\Delta t  \widehat L_N}\widehat u_N(t_n)+\Delta t  \textstyle\sum\limits_{i=1}^s b_i(-\Delta t \widehat L_N)\nonumber\\
&&\qquad \cdot \widehat P_N\widehat g(t_n+c_i\Delta t , u(t_n+c_i\Delta t )) + \widehat \delta_{n+1},\label{222}
\end{eqnarray}
where the defect terms $\{\widehat \delta_{ni}\}_{i=1}^s$ and  $\widehat \delta_{n+1}$ are respectively given by
\begin{equation*}
\begin{split}
\widehat \delta_{ni} &= \textstyle\sum\limits_{j=1}^{r} \Delta t ^j \psi_{j,i}(-\Delta t  \widehat L_N)\widehat P_N\widehat g^{(j-1)}(t_n,u(t_n))+\widehat \delta_{ni}^{[r]},\\
\widehat \delta_{n+1} &= \textstyle\sum\limits_{i=1}^r \Delta t ^i \psi_i(-\Delta t  \widehat L_N)\widehat P_N\widehat g^{(i-1)}(t_n,u(t_n))+\widehat \delta_{n+1}^{[r]},\\
\end{split}
\end{equation*}
with the remainders $\widehat \delta_{ni}^{[r]}$ and $\widehat \delta_{n+1}^{[r]}$ defined respectively by
\begin{equation*}
\begin{split}
\widehat \delta_{ni}^{[r]}=&\int_0 ^{c_i \Delta t }e^{-(c_i\Delta t  - \tau)\widehat L_N}\int_0^{\tau}\frac{(\tau-\sigma)^{r-1}}{(r-1)!}\widehat P_N\widehat g^{(r)}(t_n+\sigma, u(t_n+\sigma))\dd \sigma \dd \tau \\
&- \Delta t \sum\limits_{k=1}^{i-1}a_{ik}(-\Delta t  \widehat L_N)\displaystyle\int_0^{c_k\Delta t }\frac{(c_k\Delta t -\sigma)^{r-1}}{(r-1)!}\widehat P_N\widehat g^{(r)}(t_n+\sigma, u(t_n+\sigma))\dd \sigma\\
&+ \int_0^{c_i \Delta t } e^{-(c_i\Delta t -\sigma)\widehat L_N}\Big(\widehat P_N\widehat g(t_n+\sigma,u_N(t_n+\sigma))-\widehat P_N\widehat g(t_n+\sigma, u(t_n+\sigma)) \Big)\dd \sigma,
\end{split}
\end{equation*}
\begin{equation*}
\begin{split}
\widehat \delta_{n+1}^{[r]} =& \int_0 ^{\Delta t }e^{-(\Delta t  - \tau)\widehat L_N}\int_0^{\tau}\frac{(\tau-\sigma)^{r-1}}{(r-1)!}\widehat P_N\widehat g^{(r)}(t_n+\sigma, u(t_n+\sigma))\dd \sigma \dd \tau \\
&- \Delta t  \sum\limits_{i=1}^s b_i(-\Delta t  \widehat L_N)\displaystyle\int_0^{c_i\Delta t }\frac{(c_i\Delta t -\sigma)^{r-1}}{(r-1)!}\widehat P_N\widehat g^{(r)}(t_n+\sigma,u(t_n+\sigma))\dd \sigma\\
&+ \int_0^{\Delta t }e^{-(\Delta t -\sigma)\widehat L_N}\Big(\widehat P_N\widehat g(t_n+\sigma, u_N(t_n+\sigma))-\widehat P_N\widehat g(t_n+\sigma, u(t_n+\sigma)) \Big)\dd \sigma.
\end{split}
\end{equation*}
Here $r$ can be any nonnegative integers such that $g^{(r)}(t, u(t))$  exists and is continuous. In the same way, the defect terms defined on the time domain is shown as follows:
\begin{equation*}
\begin{split}
\delta_{ni} &= \textstyle\sum\limits_{j=1}^{r} \Delta t ^j \psi_{j,i}(-\Delta t L_N) P_N g^{(j-1)}(t_n,u(t_n))+\delta_{ni}^{[r]},\\
\delta_{n+1} &= \textstyle\sum\limits_{i=1}^r \Delta t ^i \psi_i(-\Delta t  L_N) P_N g^{(i-1)}(t_n,u(t_n))+\delta_{n+1}^{[r]},\\
\end{split}
\end{equation*}
with the remainders $\delta_{ni}^{[r]}$ and $\delta_{n+1}^{[r]}$ defined respectively by
\begin{equation*}
\begin{split}
\delta_{ni}^{[r]}=&\int_0 ^{c_i \Delta t }e^{-(c_i\Delta t  - \tau)L_N}\int_0^{\tau}\frac{(\tau-\sigma)^{r-1}}{(r-1)!}P_N g^{(r)}(t_n+\sigma, u(t_n+\sigma))\dd \sigma \dd \tau \\
&- \Delta t \sum\limits_{k=1}^{i-1}a_{ik}(-\Delta t  L_N)\displaystyle\int_0^{c_k\Delta t }\frac{(c_k\Delta t -\sigma)^{r-1}}{(r-1)!}P_N g^{(r)}(t_n+\sigma, u(t_n+\sigma))\dd \sigma\\
&+ \int_0^{c_i \Delta t } e^{-(c_i\Delta t -\sigma) L_N}\Big(P_N g(t_n+\sigma,u_N(t_n+\sigma))-P_N g(t_n+\sigma, u(t_n+\sigma)) \Big)\dd \sigma,
\end{split}
\end{equation*}
\begin{equation*}
\begin{split}
\delta_{n+1}^{[r]} =& \int_0 ^{\Delta t }e^{-(\Delta t  - \tau) L_N}\int_0^{\tau}\frac{(\tau-\sigma)^{r-1}}{(r-1)!}P_Ng^{(r)}(t_n+\sigma, u(t_n+\sigma))\dd \sigma \dd \tau \\
&- \Delta t  \sum\limits_{i=1}^s b_i(-\Delta t L_N)\displaystyle\int_0^{c_i\Delta t }\frac{(c_i\Delta t -\sigma)^{r-1}}{(r-1)!}P_N g^{(r)}(t_n+\sigma,u(t_n+\sigma))\dd \sigma\\
&+ \int_0^{\Delta t }e^{-(\Delta t -\sigma) L_N}\Big(P_Ng(t_n+\sigma, u_N(t_n+\sigma))- P_N g(t_n+\sigma, u(t_n+\sigma)) \Big)\dd \sigma.
\end{split}
\end{equation*}

In what follows, we will adopt the arguments proposed in  \cite{HochbruckOstermann2005b} to bound $\|u_N^n-u_N(t_n)\|_2$. For brevity, let us define  $\widehat e_n = \widehat u_N^n-\widehat u_N(t_n), e_n=u_N^n-u_N(t_n)$
and  $\widehat E_{ni}=\widehat U_{ni}-\widehat u_N(t_n+c_i\Delta t ),E_{ni}=U_{ni}-u_N(t_n+c_i\Delta t)$ for $i=1,\cdots,s$.
Then we arrive at the following recurrence relations:
\begin{eqnarray}
\widehat E_{ni} =& e^{-c_i\Delta t \widehat L_N}\widehat e_n+
\Delta t \textstyle\sum\limits_{j=1}^{i-1} a_{ij}(-\Delta t \widehat L_N)\Big(\widehat P_N\widehat g(t_n+c_j\Delta t ,U_{nj})\nonumber\\
&\quad-\widehat P_N\widehat g(t_n+c_j\Delta t ,u(t_n+c_j\Delta t ))\Big)-\widehat \delta_{ni}. \label{recursion-E}\\
\widehat e_{n+1} =& e^{-\Delta t \widehat L_N}\widehat e_n+\Delta t \textstyle\sum\limits_{i=1}^s b_i(-\Delta t \widehat L_N)\Big(\widehat P_N\widehat g(t_n+c_i\Delta t ,U_{ni})
\nonumber\\&\quad -\widehat P_N\widehat g(t_n+c_i\Delta t ,u(t_n+c_i\Delta t ))\Big)-\widehat \delta_{n+1},\label{recursion-e}
\end{eqnarray}

We first have the following result on the defect terms in \eqref{111} and \eqref{222}.
%%%%%%%%%%%%%%%%%%%%%%%%%% Maybe the provement can be skipped %%%%%%%%%%%%%%%%%%%%%%%%%%%%%%%%%%
% The following lemmas can be derived by replacing the norm in physical space with frequency space in \cite{HuangJu2023} due to the norm equivalence. Then the corresponding error estimate is obviously held.

\begin{lemma}
	\label{lemma6}  Given an integer $r=1$ or $2$.
	Suppose that the function $f$ satisfies Assumptions \ref{assumption4} and \ref{assumption5}, and  the exact solution $u(t)$ fulfills \eqref{regularity1} and \eqref{regularity2} in Assumption \ref{assumption6}. Suppose that $u(t)$ additionally  fulfills \eqref{regularity3}  if $r=2$.
	Then for $n=0,\cdots,N_T$, $i = 1, \cdots, s$, it holds that
	\begin{subequations}\label{lemma4_eq}
		\begin{align}
		\textstyle\|\delta_{ni}^{[r]}\|_2 \lesssim (\Delta t )^{r+1}\textstyle\sup\limits_{0\leq \eta \leq 1}\|f^{(r)}(t_n+\eta \Delta t ,u(t_n+\eta \Delta t ),\nabla u(t_n+\eta\Delta t))\|_2 + N^{2-m},\label{lemma4-1}\\
		\textstyle\Big\|\textstyle\sum\limits_{j=0}^{n-1}e^{-j\Delta t  L_N}\delta_{n-j}^{[r]}\Big\|_2 \lesssim (\Delta t)^r \textstyle\sup\limits_{ 0 \leq t \leq T } \| f^{(r)}(t,u(t),\nabla u(t))\|_2 +N^{2-m}.\label{lemma4-2}
		\end{align}
		Note that the above hidden constants are independent of $N$ and $\Delta t$.	
	\end{subequations}
\end{lemma}
\begin{proof}
	Recalling \eqref{lemma2-1} in Lemma \ref{lemma2} and the relation \eqref{equal_norm}, we have after some direct manipulations
	\begin{equation}
	\begin{split}
	&\Big\|\widehat L_N\int_0^{c_i\Delta t }e^{-(c_i\Delta t -\tau)\widehat L_N}\int_0^{\tau}\frac{(\tau-\sigma)^{r-1}}{(r-1)!}\widehat P_N\widehat g^{(r)}(t_n+\sigma,u(t_n+\sigma))\dd \sigma \dd \tau\Big\|_{\ell^2}\\
	&\;\lesssim  \textstyle(\Delta t)^{r+1}\sup\limits_{0 \leq \tau \leq c_i\Delta t }\|e^{-(c_i\Delta t -\tau) L_N}\|_0 \textstyle\sup\limits_{0\leq \eta \leq 1}\|\widehat L_N \widehat P_N\widehat g^{(r)}(t_n+\eta \Delta t ,u(t_n+\eta \Delta t ))\|_{\ell^2}\\
	&\;\lesssim (\Delta t) ^{r+1}\textstyle\sup\limits_{0\leq \eta \leq 1}\| g^{(r)}(t_n+\eta \Delta t,u(t_n+\eta \Delta t ))\|_2\\
	&\;\lesssim (\Delta t) ^{r+1}\textstyle\sup\limits_{0\leq \eta \leq 1}\|f^{(r)}(t_n+\eta \Delta t ,u(t_n+\eta \Delta t ),\nabla u(t_n+\eta\Delta t))\|_2,
	\end{split}
	\label{mid-estimate1}
	\end{equation}
	where the last two inequality is due to that $\widehat P_N$ is $\ell^2$-stable.	Similarly, we also have
	\begin{equation}
	\begin{split}
	&\Big\|\widehat L_N\int_0^{\Delta t }e^{-(\Delta t -\sigma)\widehat L_N}\Big(\widehat P_N\widehat g(t_n+\sigma,u_N(t_n+\sigma))-\widehat P_N\widehat g(t_n+\sigma,u(t_n+\sigma)) \Big)\dd \sigma\Big\|_{\ell^2}
	\\
	&\lesssim \Big\|\int_0^{\Delta t }L_Ne^{-(\Delta t -\sigma)L_N}\dd \sigma\Big\|_0 \textstyle\sup\limits_{0 \leq \eta \leq 1} \Big\|\widehat P_N\widehat g(t_n+\eta \Delta t ,u_N(t_n+\eta \Delta t ))
	\\
	&\quad -\widehat P_N\widehat g(t_n+\eta\Delta t ,u(t_n+\eta\Delta t ))\Big\|_{\ell^2}.
	\end{split}
	\label{eq4-3}
	\end{equation}
	
	In view of the similar arguments for proving Lemma \ref{lemma3}, we have
	\begin{equation}
	\begin{split}
	\Big\|\int_0^{\Delta t }L_Ne^{-(\Delta t -\sigma)L_N}\dd \sigma \Big\|_0 &= \max\limits_{\lambda \in \lambda(L_N)} \Big|\int_0^{\Delta t }\lambda e^{-(\Delta t -\sigma)\lambda}\dd \sigma \Big| \\
	&\leq\max\limits_{\lambda \in \lambda(L_N)} |1-e^{-\Delta t\lambda}|\lesssim 1.
	\end{split}
	\label{eq4-4a}
	\end{equation}
	Since $P_N$ is $L^2$-stable and the function $\widehat g$ is locally-Lipschitz continuous (Lemma \ref{locally_Lipschitz}), we further obtain from \eqref{eq4-3}-\eqref{eq4-4a} and Theorem \ref{semi-discrete-error2} that
	\begin{equation}
	\begin{split}
	&\Big\|\widehat L_N \int_0^{\Delta t }e^{-(\Delta t -\sigma)\widehat L_N}\Big(\widehat P_N\widehat g(t_n+\sigma,u_N(t_n+\sigma))-\widehat P_N\widehat g(t_n+\sigma,u(t_n+\sigma)) \Big)\dd \sigma\Big\|_{\ell^2}
	\\
	&\lesssim \sup\limits_{0 \leq \eta \leq 1} \|\widehat P_N\widehat g(t_n+\eta \Delta t ,u_N(t_n+\eta \Delta t ))-\widehat P_N\widehat g(t_n+\eta \Delta t ,u(t_n+\eta \Delta t ))\|_{\ell^2}
	\\
	&\lesssim \sup\limits_{0 \leq \eta \leq 1}\|u_N(t_n+\eta \Delta t )-u(t_n+\eta \Delta t )\|_2\lesssim N^{2-m}.
	\end{split}
	\label{mid-estimate2}
	\end{equation}
	According to \eqref{lemma2-3} in Lemma \ref{lemma2} and the similar arguments for deriving \eqref{mid-estimate1},
	\begin{equation}
	\begin{split}
	\quad&\Big\|\widehat L_N\Delta t \sum\limits_{k=1}^{i-1} a_{ik}(-\Delta t \widehat L_N)\int_0^{c_k\Delta t }\frac{(c_k\Delta t -\sigma)^{r-1}}{(r-1)!}\widehat P_N\widehat g^{(r)}(t_n+\sigma,u(t_n+\sigma))\dd \sigma\Big\|_{\ell^2}
	\\
	&\lesssim \Big\|\Delta t \sum\limits_{k=1}^{i-1} a_{ik}(-\Delta t \widehat L_N)\int_0^{c_k\Delta t }\frac{(c_k\Delta t -\sigma)^{r-1}}{(r-1)!}\widehat L_N\\
	&\qquad\quad \cdot \widehat P_N\widehat g^{(r)}(t_n+\sigma,u(t_n+\sigma))\dd \sigma\Big\|_{\ell^2}
	\\
	&\leq  (\Delta t)^{r+1} \sum\limits_{k=1}^{i-1}\|a_{ik}(-\Delta t L_N)\|_0\sup\limits_{0 \leq \eta \leq 1}\|g^{(r)}(t_n+\eta \Delta t ,u(t_n+\eta \Delta t ))\|_2 \\
	&\lesssim (\Delta t)^{r+1}\sup\limits_{0 \leq \eta \leq 1}\| f^{(r)}(t_n+\eta \Delta t ,u(t_n+\eta \Delta t ),\nabla u(t_n+\eta\Delta t))\|_2.
	\end{split}
	\label{mid-estimate3}
	\end{equation}
	
	Now, using the triangle inequality, the regularity assumptions for $u(t)$ and $f$, and the estimates \eqref{mid-estimate1}, \eqref{mid-estimate2} and \eqref{mid-estimate3}, we get
	\begin{equation}
	\|\delta_{ni}^{[r]}\|_2 \lesssim (\Delta t )^{r+1}\sup\limits_{0 \leq \eta \leq 1}\|f^{(r)}(t_n+\eta \Delta t ,u(t_n+\eta \Delta t ),\nabla u(t_n+\eta\Delta t))\|_2+ N^{2-m},\quad\forall\, i=1, \cdots, s,
	\end{equation}
	which leads to \eqref{lemma4-1}.  Also \eqref{lemma4-2} can be derived in the similar manner.
	
\end{proof}

\begin{lemma}
	\label{lemma8}
	Suppose the function $f$ satisfies Assumptions \ref{assumption4} and \ref{assumption5}, and  the exact solution $u(t)$ fulfills \eqref{regularity1}-\eqref{regularity2}. If $s\geq 2$, then it holds for any  $0 \leq n < N_T$,
	\begin{equation}
	\label{E_n}
	\|E_{ni}\|_2 \lesssim \|e_n\|_2 + (\Delta t)^2 \textstyle\sup\limits_{0\leq \eta \leq 1} \|f'(t_n+\eta \Delta t ,u(t_n+\eta \Delta t ), \nabla u(t_n+\eta\Delta t)\|_2 + N^{2-m}, \quad \forall\, i=1, \cdots, s,
	\end{equation}
	where the hidden constant is independent of $N$ and $\Delta t$.
\end{lemma}

\begin{proof}
	According to the definition of $\psi_{j,i}$ in (\ref{function}), we have by some manipulations that $\psi_{1, j}=0, j=1, \cdots, s$ when the EIFG method \ref{explicit_ETD} is consistent. Therefore, the estimation of $\|\delta_{ni}\|_2$ can be converted to that of $\|\delta_{ni}^{[1]}\|_2$. Using the similar arguments for deriving the estimate \eqref{mid-estimate2} and \eqref{lemma2-3} in Lemma \ref{lemma2}, we have
	\begin{align}
	\begin{split} &\Big\|\widehat L_N\Delta t \textstyle\sum\limits_{j=1}^{i-1}a_{ij}(-\Delta t \widehat L_N)\Big(\widehat P_N\widehat g(t_n+c_j\Delta t ,U_{nj})-\widehat P_N\widehat g(t_n+c_j\Delta t ,u(t_n+c_j\Delta t )) \Big)\Big\|_{\ell^2}\\
	&\lesssim  \textstyle\sum\limits_{j=1}^{i-1}\|\Delta t L_Na_{ij}(-\Delta t L_N)\|_0\nonumber\\
	&\qquad \cdot\textstyle\max\limits_{2 \leq j \leq i -1} \|\widehat P_N\widehat g(t_n+c_j\Delta t ,U_{nj})-\widehat P_N\widehat g(t_n+c_j\Delta t ,u(t_n+c_j\Delta t ))\|_{\ell^2}\\
	&\lesssim \textstyle\max\limits_{2 \leq j \leq i -1} \|\widehat P_N\widehat g(t_n+c_j\Delta t ,U_{nj})-\widehat P_N\widehat g(t_n+c_j\Delta t ,u(t_n+c_j\Delta t ))\|_{\ell^2}\\
	&\lesssim  \textstyle\max\limits_{2 \leq j \leq i-1}\|E_{nj}\|_2+N^{2-m},
	\end{split}
	\end{align}
	where the last inequality uses that $\widehat g$ is locally-Lipschitz continuous.
	
	Note that $\|\delta_{ni}^{[1]}\|_2$ is uniformly bounded for $i = 1, \cdots, s$ (see \eqref{lemma4-1} in Lemma \ref{lemma4} with $r=1$). Recalling the relation \eqref{recursion-E}, we have by the triangle inequality that {
		\begin{equation*}
		\begin{split}
		&\|E_{ni}\|_2=\|L_N\widehat E_{ni}\|_{\ell^2}\\
		& \lesssim\; \|\widehat L_Ne^{-c_i\Delta t \widehat L_N}\widehat e_n\|_{\ell^2}+(\Delta t)^2\textstyle\sup\limits_{0 \leq \eta \leq 1}\|\widehat L_N\widehat P_N\widehat g'(t_n+\eta\Delta t ,u(t_n+\eta \Delta t ))\|_{\ell^2}+N^{2-m}\\
		&\quad+\Big\|\widehat L_N\Delta t  \textstyle\sum\limits_{j=1}^{i-1}a_{ij}(-\Delta t \widehat L_N)\Big(\widehat P_N\widehat g(t_n+c_j\Delta t ,U_{nj})-\widehat P_N\widehat g(t_n+c_j\Delta t ,u(t_n+c_j\Delta t )) \Big)\Big\|_{\ell^2}\\
		&\lesssim\; \|e_n\|_2 +  (\Delta t)^2 \textstyle\sup\limits_{0\leq \eta \leq 1} \| g'(t_n+\eta \Delta t , u(t_n+\eta \Delta t ))\|_2+ N^{2-m}+ \textstyle\max\limits_{2 \leq j \leq i -1}\|E_{nj}\|_2\\
		&\lesssim \; \|e_n\|_2 +  (\Delta t)^2 \textstyle\sup\limits_{0\leq \eta \leq 1}\| f'(t_n+\eta \Delta t , u(t_n+\eta \Delta t ),\nabla u(t_n+\eta\Delta t))\|_2 + N^{2-m}
		+ \textstyle\max\limits_{2 \leq j \leq i -1} \|E_{nj}\|_2.
		\end{split}
		\end{equation*}}
	Finally \eqref{E_n} is obtained by recursively using the above inequality.
\end{proof}

\begin{theorem}[Error estimate for the EIFG2 scheme] \label{thm4}
	Suppose that the function $f$ satisfies Assumptions \ref{assumption4} and \ref{assumption5}, and the exact solution $u(t)$ fulfills \eqref{regularity1}-\eqref{regularity3} in Assumptions \ref{assumption4}. There exists a constant $N_0>0$ such that if the number of spatial meshes $N \geq N_0$, then after Fourier transform, the numerical solution $\{u_N^n\}$ produced by EIFG2 scheme  \eqref{EIFE2} satisfies
	\begin{equation}\label{err2}
	\|u(t_n)-u_N^n\|_2 \lesssim (\Delta t)^2 + N^{2-m},\quad \forall\,n=1,\cdots, N_T,
	\end{equation}
	where the hidden constant is independent of $N$ and $\Delta t$.	
\end{theorem}

\begin{proof}	
	Recalling the definition of $\psi_i$ in (\ref{function}), we can check that $\psi_{1}(-\Delta t L_N)=\psi_2(-\Delta t L_N)=0$, which implies that $\delta_{n+1}=\delta_{n+1}^{[2]}$ by Lagrangian interpolation theorem for $s = 2$. By \eqref{lemma2-3} in Lemma \ref{lemma2}, we have
	\begin{equation}
	\label{eq4-4}
	\begin{split}
	&\Big\|\widehat L_N\Delta t \textstyle\sum\limits_{j=0}^{n-1} e^{-(n-1-j)\Delta t \widehat L_N}\textstyle\sum\limits_{i=1}^s b_i(-\Delta t \widehat L_N)\Big(\widehat P_N\widehat g(t_j+c_i\Delta t , U_{ji})-\widehat P_N\widehat g(t_j+c_i\Delta t ,u(t_j+c_i\Delta t )) \Big)\Big\|_{\ell^2}
	\\
	&\lesssim  \Big\|\widehat L_N\Delta t \textstyle\sum\limits_{i=1}^s b_i(-\Delta t \widehat L_N)\Big(\widehat P_N\widehat g(t_{n-1}+c_i\Delta t ,U_{n-1,i})-\widehat P_N\widehat g(t_{n-1}+c_i\Delta t ,u(t_{n-1}+c_i\Delta t ) \Big)\Big\|_{\ell^2}
	\\
	&\quad+ \Big\|\widehat L_N\Delta t \textstyle\sum\limits_{j=0}^{n-2} e^{-(n-1-j)\Delta t \widehat L_N}\textstyle\sum\limits_{i=1}^s \Big(\widehat P_N\widehat g(t_j+c_i\Delta t ,U_{ji})-\widehat P_N\widehat g(t_j+c_i\Delta t ,u(t_j+c_i\Delta t )) \Big)\Big\|_{\ell^2}\\
	&=: {\rm II}_1+{\rm II}_2.
	\end{split}
	\end{equation}
	Following the similar arguments for \eqref{mid-estimate2} in Lemma \ref{lemma8},  since \eqref{lemma2-2} and \eqref{lemma2-3} in Lemma \ref{lemma2},  we can obtain
	\begin{equation}
	\label{add-1}
	\begin{split}
	{\rm II}_1\lesssim \,& \textstyle\sum\limits_{i=1}^s \Big\|\Delta t L_Nb_i(-\Delta t L_N)\Big\|_0\textstyle\max\limits_{1\le i\le s}\|\widehat P_N\widehat g(t_{n-1}+c_i\Delta t ,U_{n-1,i})\\
	&\qquad\qquad-\widehat P_N\widehat g(t_{n-1}+c_i\Delta t ,u(t_{n-1}+c_i\Delta t ))\|_{\ell^2}\\
	\lesssim \,& \textstyle\max\limits_{1\le i\le s} \|\widehat P_N\widehat g(t_{n-1}+c_i\Delta t ,U_{n-1,i})-\widehat P_N\widehat g(t_{n-1}+c_i\Delta t ,u(t_{n-1}+c_i\Delta t ))\|_{\ell^2} \\
	\lesssim \,& \textstyle\max\limits_{1 \leq i \leq s} \|E_{n-1,i}\|_2+N^{2-m},
	\end{split}
	\end{equation}
	and
	\begin{equation}
	\label{add-2}
	\begin{split}
	{\rm II}_2
	\lesssim \,& \Big\|\Delta t L_N^{\frac 1 2}\textstyle\sum\limits_{j=0}^{n-2}e^{-(n-1-j)\Delta t L_N}\Big\|_0\textstyle\sup\limits_{ 0  \leq t \leq    T }\Big\|\widehat L_N^{\frac 1 2}\Big(\widehat P_N\widehat g(t,u(t))-\widehat P_N\widehat g(t,u_N(t))\Big)\Big\|_{\ell^2} \\&+ \textstyle\sum\limits_{j=0}^{n-2} \Delta t \Big\|L_N^{\frac 1 2} e^{-(n-j-1)\Delta t L_N}\Big\|_0\textstyle\max\limits_{1\le i \le s} \Big\|\widehat L_N^{\frac 1 2}\Big(\widehat P_N\widehat g(t_j+c_i\Delta t ,U_{ji})\\
	&\qquad-\widehat P_N\widehat g(t_j+c_i\Delta t ,u_N(t_j+c_i\Delta t ))\Big)\Big\|_{\ell^2}
	\\
	\lesssim\,& \Delta t\textstyle\sum\limits_{j=0}^{n-2}t_{n-j-1}^{-\frac 1 2}\textstyle\max\limits_{1\le i \le s}\Big\|\widehat L_N^{\frac 1 2}\Big(\widehat P_N\widehat g(t_j+c_i\Delta t ,U_{ji})\\
	&\qquad-\widehat P_N\widehat g(t_j+c_i\Delta t ,u_N(t_j+c_i\Delta t ))\Big)\Big\|_{\ell^2}+N^{2-m}
	\\
	\lesssim \,& \Delta t\textstyle\sum\limits_{j=0}^{n-2}t_{n-j-1}^{-\frac 1 2}\textstyle\max\limits_{1 \leq i \leq s} \|E_{ji}\|_2+N^{2-m}.
	\end{split}
	\end{equation}
	
	With the help of the estimates \eqref{eq4-4}, \eqref{add-1} and \eqref{add-2}, it follows from the relation \eqref{recursion-e} and \eqref{lemma4-2} in Lemma \ref{lemma4} (with $r=2$) that
	\begin{equation*}
	\begin{split}
	\|e_n\|_2 \lesssim\,&  \Big\|\widehat L_N\Delta t \textstyle\sum\limits_{j=0}^{n-1} e^{-(n-1-j)\Delta t \widehat L_N}\textstyle\sum\limits_{i=1}^s b_i(-\Delta t \widehat L_N)\Big(\widehat P_N\widehat g(t_j+c_i\Delta t , U_{ji})\nonumber\\
	&\;-\widehat P_N\widehat g(t_j+c_i\Delta t ,u(t_j+c_i\Delta t )) \Big)\Big\|_{\ell^2}+\Big\|\widehat L_N\textstyle\sum\limits_{j=0}^{n-1}e^{-j\Delta t \widehat L_N}\delta_{n-j}^{[2]}\Big\|_{\ell^2}
	\\
	\lesssim \,& \Delta t\textstyle\max\limits_{1\le i \le s}  \|E_{n-1,i}\|_2
	+\Delta t\textstyle\sum\limits_{j=0}^{n-2}t_{n-j-1}^{-\frac 1 2}\textstyle\max\limits_{1\le i \le s}\|E_{ji}\|_2\\
	&
	+\Delta t^2\textstyle\sup\limits_{ 0  \leq t \leq    T }\| f^{(2)}(t,u(t),\nabla u(t))\|_2+N^{2-m}
	\\
	\lesssim \,& \Delta t\textstyle\sum\limits_{j=0}^{n-1}t_{n-j-1}^{-\frac 1 2}\textstyle\max\limits_{1\le i \le s} \|E_{ji}\|_2+(\Delta t)^2+N^{2-m}.
	\end{split}
	\end{equation*}
	This combined with the estimation of $\|E_{ji}\|_2$ in Lemma \ref{lemma8} and the discrete Gronwall inequality leads to
	\begin{equation}\label{err2c}
	\|u_N^n- u_N(t_n)\|_2 \lesssim (\Delta t)^2 + N^{2-m}.
	\end{equation}
	Finally, the combination of \eqref{fully-discrete-error}, \eqref{semi-discrete-error-eq} and \eqref{err2c} immediately gives  \eqref{err2}.	
\end{proof}

\begin{remark}
	When the reaction term only depends on $t$ and $u(t)$, then the error estimate can be derived in $H^1$-norm and the corresponding spatial accuracy can be increased to order $N^{1-m}$ and the temporal accuracy is of order $\Delta t^2$ for EIFG2 method, i.e.,
	\begin{equation}
	\label{remark1}
	\|u(t_n)-u_N^n\|_1 \lesssim (\Delta t)^2 + N^{1-m},\quad \forall\,n=1,\cdots, N_T,
	\end{equation}
	and the above estimation can be easily derived by following similar arguments in \cite{HuangJu2023}.
\end{remark}

\begin{remark}
	The temporal convergence order of EIFG method can be improved when choosing $s \geq 3$. But the theoretical analysis will be quite complicated since the estimates of $\|e_n\|_2$ and $\|E_{ni}\|_2$ will be coupled together. We refer the reader to \cite{HochbruckOstermann2005b} for some details along this line, and rigorous error estimates of higher-order EIFG schemes would be an interesting open question.
\end{remark}

%\begin{remark}
%	Due to the norm equivalence between frequency space and physical space, all conclusions above can also be held on the physical space, i.e., the semi-discrete error proposed in Theorem \ref{semi-discrete-error2} can be converted to 
%	\begin{equation*}
%	\|u(t)-u_N(t)\|_0+N^{-1}\|u(t)-u_N(t)\|_1 \lesssim N^{-m}, \quad t \in [0, T],
%	\end{equation*}
%	and the fully-discrete error for EIFG1 and EIFG2 method (proposed in Theorem \ref{thm7} and Theorem \ref{thm4}) can be reformulated as
%	\begin{equation*}
%	\|u(t_n)-u_N^n\|_1 \lesssim \Delta t^r+N^{1-m}, \quad \forall \ n=1, \cdots, N_T,
%	\end{equation*}
%	where $r=1$ for EIFG1 method and $r=2$ for EIFG2 method.
%\end{remark}

\section{Numerical experiments and applications}\label{numerical}

In this section, several numerical examples and applications are presented to illustrate the performance of the proposed efficient EIFG method and then verify the convergent order of numerical scheme we have derived in Section \ref{theory}. All numerical experiments are done using Matlab on an Intel i5-8250U, 1.80GHz CPU laptop with 8GB memory. Specially, we choose the interpolation node $c_2=\frac 1 2$ for EIFG2 method. Moreover, to illustrate the great performance of EIFG, we also present numerical schemes with third-order temporal convergence, which are abbreviated as EIFG3. The corresponding exponential Runge-Kutta tableau is shown in the Table \ref{tab 5} as \cite{Krogstad2005},
\begin{table}[h]
	\centering
	\caption{The exponential Runge-Kutta tableau of EIFG3}
	\begin{tabular}{c|cccc}
		0 &  &  &  & \\
		$\frac 1 2$ & $\frac 1 2 \varphi_{1,2}$ &  & & \\
		$\frac 1 2$ & $\frac 1 2 \varphi_{1,3}-\varphi_{2,3}$ & $\varphi_{2,3}$ & & \\
		1 & $\varphi_{1,4}-2\varphi_{2,4}$ & 0 & $2\varphi_{2,4}$ & \\
		\hline
		& $\varphi_1-3\varphi_2+4\varphi_3$ & $2\varphi_2-4\varphi_3$ & $2\varphi_2-4\varphi_3$ & $-\varphi_2+4\varphi_3$
	\end{tabular}
	\label{tab 5}
\end{table}
where $\varphi_{j,k}=\varphi_j(-c_k\Delta tL_N)$ and $\varphi_j = \varphi_j(-\Delta tL_N)$.

\subsection{Convergence tests}
We first verify the error estimation obtained in Theorem \ref{thm4} for EIFG2 scheme and EIFG3 scheme at the terminal time $T$. Since the reaction term in Example \ref{ex1} and Example \ref{ex2} are independent of $\nabla u$, we only discuss the numerical accuracy in $H^1$-norm.

\begin{example}
	\label{ex1}
	In this example, the three-dimensional nonlinear reaction diffusion problem with periodic boundary condition is shown as follows:
	\begin{equation*}
	\left\{\begin{split}
	&u_t= \Delta u -u +f(t,x,y,z), \quad (x,y,z) \in \Omega,0 \leq t \leq T\\
	&u(0, x, y,z) = x^2(x-1)^2\sin(2\pi x)y^2(y-1)^2\sin(2\pi y)z^2(z-1)^2\sin(2\pi z), \quad (x,y,z) \in \Omega,
	\end{split}  \right.
	\end{equation*}
	where $\Omega = [0,1]^3$, $f(t,x,y,z)$ is determined by the exact solution and the terminal time $T=1$. The exact solution is given by $u(t,x,y)=e^{-t}x^2(x-1)^2\sin(2\pi x)y^2(y-1)^2\sin(2\pi y)z^2(z-1)^2\sin(2\pi z)$.
	
\end{example}

For the spatial accuracy tests, we run the EIFG2 scheme with fixed temporal partition $N_T=4096$ (i.e. $\Delta T=T/N_T=1/4096$) and uniformly refined spatial grids with $N_x \times N_y\times N_z=8\times 8\times 8, 16 \times 16\times 16, 32 \times 32\times 32$ and $64 \times 64\times 64$, respectively, so it's obvious that the temporal step size is much finer than spatial mesh size. Meanwhile, for the temporal accuracy tests, we fix the spatial grids with $N_x \times N_y\times N_z=256 \times 256\times 256$, and the temporal partitions are $N_T=4,8,16,32$ for EIFG2 and EIFG3 method. All numerical results are reported in Table \ref{tab 1}, including the numerical errors measured in $L^2$ and $H^1$ norms and the corresponding convergence rates. Observing from the numerical results, we can observe the roughly fourth-order spatial convergence with respect to $L^2$ norm as expected but a half order higher than third-order convergence in $H^1$ norm. It's also easy to observe the second-order temporal convergence for EIFG2 scheme and third-order temporal convergence for EIFG3 scheme, which coincide well with the error estimates derived in Theorem \ref{thm4}.

Table \ref{tab 2} reports the average CPU time costs (seconds) per iteration for the EIFG2 scheme and corresponding growth factors along the refinement of the spatial mesh. All tests are run with fixed temporal partitions $N_T=50$. The uniform spatial meshes are fixed with $N_x \times N_y\times N_z=16 \times 16\times 16, 32 \times 32\times 32, 64\times 64\times 64$ and $128\times 128\times 128$. The results clearly show that the computational cost grows almost linearly along with the number of mesh nodes, which matches well with the property of FFT and demonstrates the high efficiency of our EIFG method.

\begin{table}[htbp]
	\centering
	\caption{Numerical results on the solution errors measured in the $L^2$ and $H^1$ norms and corresponding convergence rates for the EIFG2 and EIFG3 schemes in Example \ref{ex1}.}
	\begin{tabular}{|cccccc|}
		\hline
		$N_T$ & $N_x\times N_y\times N_z$& $\|u_N^n-u(T_n)\|_0$ & CR & $\|u_N^n-u(T_n)\|_1$ & CR\\
		\hline
		\multicolumn{6}{|c|}{Spatial accuracy tests for EIFG2}\\
		\hline
		4096 & $8 \times 8\times 8$ & 7.8049e-08 & - & 1.2119e-06 & - \\
		4096 & $16 \times 16\times 16$ & 3.6115e-09 & 4.43 & 8.1247e-08 & 3.90 \\
		4096 & $32 \times 32\times 32$ & 2.0478e-10 & 4.14 & 6.8708e-09 & 3.56 \\
		4096 & $64 \times 64\times 64$ & 1.2213e-11 & 4.07 & 6.0627e-10 & 3.50\\
		\hline
%		\multicolumn{6}{|c|}{Temporal accuracy tests for EIFG1}\\
%		\hline
%		8 &$256 \times 256 \times 256$ & 8.5799e-07 & - & 1.1533e-05 & - \\
%		16 &$256 \times 256 \times 256$ & 3.9309e-07 & 1.13 & 5.3291e-06 & 1.11 \\
%		32 &$256 \times 256 \times 256$ & 1.7432e-07 & 1.17 & 2.3969e-06 & 1.15 \\
%		64 &$256 \times 256 \times 256$ & 7.5059e-08 & 1.22 & 1.0395e-06 & 1.21\\
%		\hline
		\multicolumn{6}{|c|}{Temporal accuracy tests for EIFG2}\\
		\hline
		4 & $256 \times 256 \times 256$ & 1.2555e-07 & - & 1.6657e-06 & -\\
		8 & $256 \times 256 \times 256$ & 2.9656e-08 & 2.08 & 3.9370e-07 & 2.08\\
		16 & $256 \times 256 \times 256$ & 7.1061e-09 & 2.06 & 9.4935e-08 & 2.05\\
		32 & $256 \times 256 \times 256$ & 1.6441e-09 & 2.11 & 2.2386e-08 & 2.09\\
		\hline
		\multicolumn{6}{|c|}{Temporal accuracy tests for EIFG3}\\
		\hline
		4 & $256 \times 256 \times 256$ & 3.8449e-10& -& 4.6395e-09&- \\
		8 & $256 \times 256 \times 256$ & 6.0865e-11& 2.67& 7.6455e-10& 2.60\\
		16 & $256 \times 256 \times 256$ & 7.0733e-12& 3.11& 9.8254e-11& 2.96\\
		32 & $256 \times 256 \times 256$ & 5.3509e-13& 3.72& 8.8474e-12& 3.47\\
		\hline
	\end{tabular}
	\label{tab 1}
\end{table}

\begin{table}
	\centering
	\caption{The average CPU time costs (seconds) per iteration under different spatial meshes and corresponding growth factors with respect to the number of mesh nodes for the EIFG2 scheme in Example \ref{ex1}}
	\begin{tabular}{|cccc|}
		\hline
		$N_x \times N_y \times N_z$ & $N_T$ & Average CPU time  & Growth factor\\
		& & cost per step&\\
		\hline
		$16\times 16 \times 16$ & 50 & 0.02715 & - \\
		$32\times 32 \times 32$ & 50 & 0.22478 & 1.016 \\
		$64\times 64 \times 64$ & 50 & 1.81191 & 1.004 \\
		$128 \times 128 \times 128$ & 50 & 14.01884 & 0.984\\
		\hline
	\end{tabular}
	\label{tab 2}
\end{table}

\subsection{Mean Curvature Flow}

\begin{example}\label{ex3}
	In this example, we consider the mean curvature flow problem \cite{EvanSpruck1999,MottoniSchatzman1995}. Let $\Omega=[-0.5,0.5]^d$ where $d>1$ denotes the dimension of the space. The problem we simulate is shown in the following:
\begin{equation*}
	\left\{\begin{split}
	&u_t= \Delta u -\frac{1}{\epsilon^2} (u^3-u), \quad \bm x \in \Omega ,\quad \ 0 \leq t \leq T,\\
	&u(0, x, y, z) = \tanh \Big( \frac{R_0-\|\bm x\|_0}{\sqrt{2}\epsilon}\Big), \quad \bm x \in \Omega, \\
	\end{split}  \right.
	\end{equation*}
	where $R_0=0.4$ and the terminal time is taken to be $T=0.075$.
\end{example}

The above example has been widely used in many works, such as \cite{ChenShen1998,DuLiu2004,DuZhu2005,FengProhl2003,XuTang2006,ZhangDu2009}. Suppose the case is equipped with periodic boundary condition, the problem describes the shrinking process of a circle in 2D or a sphere in 3D. Denote $R(t)$ as the radius of the circular region at time $t$ in 2D and $V(t)$ as the volume of the sphere in 3D. It has been proved that when $\epsilon \to 0$, the theoretic limit radius $R_{lim}(t)$ satisfies \cite{ZhangDu2009,LiLee2010}
\begin{equation*}
\frac{\dd R_{lim}}{\dd t}=\frac{1-d}{R_{lim}}.
\end{equation*}

Thus, we have
\begin{equation*}
R_{lim}(t)=\sqrt{R_0^2+2(1-d)t}.
\end{equation*}
Correspondingly it holds that
\begin{equation*}
\left\{\begin{split}
&V_{lim}(t)=\pi (R_0^2-2t), \quad d=2,\\
&V_{lim}^{\frac 2 3}(t)=\Big(\frac 4 3\pi \Big)^{\frac 2 3}(R_0^2-4t), \quad d=3.
\end{split}
\right.
\end{equation*}

We simulate all numerical tests with the interface thickness $\epsilon=0.067, 0.075$, respectively. Since there's no closed form for the exact solution, we only test the temporal accuracy and regard the solution obtained with finest grids $N_x\times N_y=2048\times 2048$ as the approximate exact solution for temporal accuracy tests and denote the corresponding radius as $R_{\epsilon}$. We fix the uniform spatial meshes with $N_x\times N_y = 2048 \times 2048$ and the temporal partitions $N_T=32,64,128,256$ for EIFG2 method, and $N_T=8,16,32,64$ for EIFG3 method. Also, the approximate exact solution for temporal accuracy tests is obtained by finest temporal partitions $N_T=1024$. All numerical results are shown in Table \ref{tab 8}, including the radius error of the circle and the corresponding rates. We can also find the second-order rate for EIFG2 scheme and third-order rate for EIFG3 scheme, which are consistent with the theoretical results derived in Theorem \ref{thm4}. Moreover, we present the circle shrinking process in the 2D space, which are obtained by simulating EIFG2 method with $\epsilon=0.05,N_x\times N_y=2048\times 2048$ and $N_T=1024$. For testing the sphere shrinking process in the 3D space, we simulate the EIFG2 method with $N_x\times N_y\times N_z=256\times 256\times 256,N_T=1024$ and set $T=\frac{3}{\sqrt{2}\epsilon}$. Observing from Figure \ref{fig 5} and Figure \ref{fig 6}, we can obtain that the radius of circles and spheres are both monotonically decreasing along the time interval.

\begin{table}[htbp]
	\centering
	\caption{Errors and convergence rates on the radius of the shrinking circle at the final time $T$ of Example \ref{ex3} by using EIFG2 and EIFG3 methods.}
	\begin{tabular}{|cccccc|}
		\hline
		$N_T$&$N_x\times N_y$ &\multicolumn{2}{c}{$\epsilon=0.067$}&\multicolumn{2}{c|}{$\epsilon=0.075$}\\
		\cline{3-6}
		 & &  $|R-R_{\epsilon}|$ & CR & $|R-R_{\epsilon}|$ & CR\\
		\hline
%		\multicolumn{6}{|c|}{Temporal accuracy tests for EIFG1}\\
%		\hline
%		64 &$2048\times 2048$ & 6.7871e-02 & - & 2.5527e-02 & -\\
%		128 &$2048\times 2048$ & 3.8574e-02 & 0.82 & 1.2707e-02 & 1.01\\
%		256 &$2048\times 2048$ & 1.9043e-02 & 1.02 & 5.6338e-03 & 1.17\\
%		512 &$2048\times 2048$ & 6.8359e-03 & 1.48& 1.9071e-03 & 1.56\\
%		\hline
		\multicolumn{6}{|c|}{Temporal accuracy tests for EIFG2}\\
		\hline
		32 & $2048\times 2048$ & 1.3353e-02 & - & 6.1547e-03 & -\\
		64 & $2048\times 2048$ & 3.9190e-03 & 1.77& 1.6849e-03 & 1.87\\
		128 & $2048\times 2048$ & 1.0386e-03 & 1.92& 4.4831e-04 & 1.91\\
		256 & $2048\times 2048$ & 2.6230e-03 & 1.99& 1.1476e-04 & 1.97\\
		\hline
		\multicolumn{6}{|c|}{Temporal accuracy tests for EIFG3}\\
		\hline
		8 & $2048\times 2048$ & 8.5600e-03 & - & 3.1555e-03 & -\\
		16 & $2048\times 2048$ & 1.4073e-03 & 2.60& 6.7503e-04 & 2.22\\
		32 & $2048\times 2048$ & 1.4168e-04 & 3.31& 7.9937e-05 & 3.08\\
		64 & $2048\times 2048$ & 3.3353e-05 & 2.09& 6.1505e-06 & 3.70\\
		\hline
	\end{tabular}
	\label{tab 8}
\end{table}

\begin{figure}[htbp]
	\centering
	\subfigure{
		\label{fig5-1}
		\centering
		\includegraphics[width = 120pt,height=120pt]{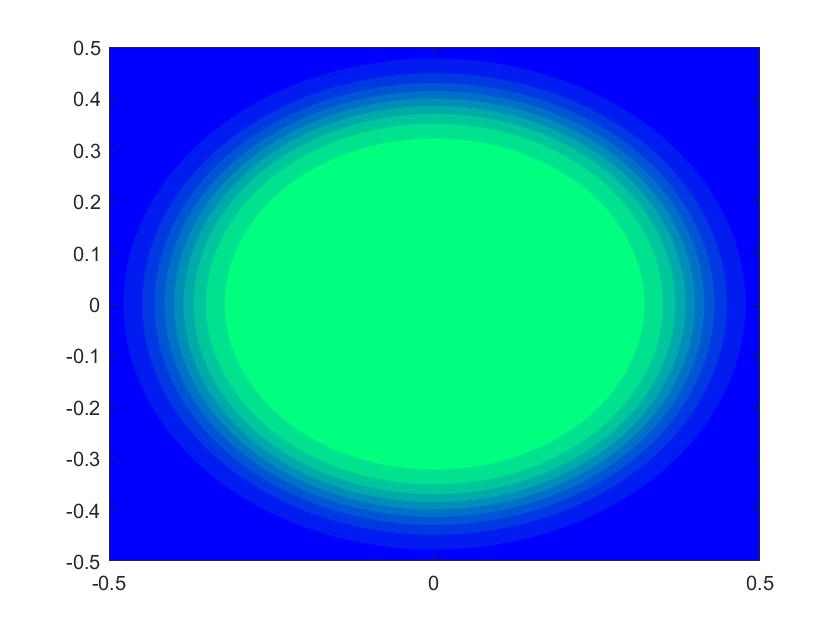}
	}
	\subfigure{
		\label{fig5-2}
		\centering
		\includegraphics[width = 120pt,height=120pt]{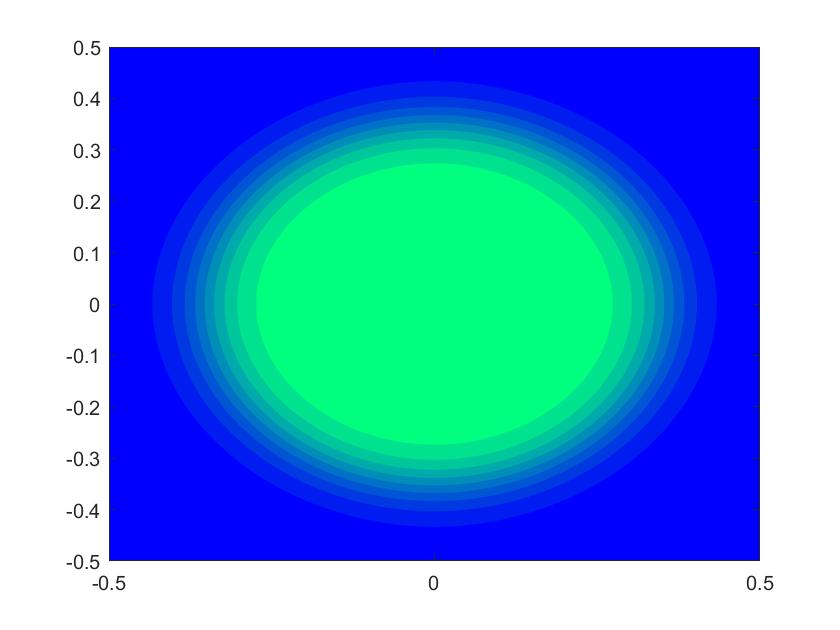}
	}
	\subfigure{
		\label{fig5-3}
		\centering
		\includegraphics[width = 120pt,height=120pt]{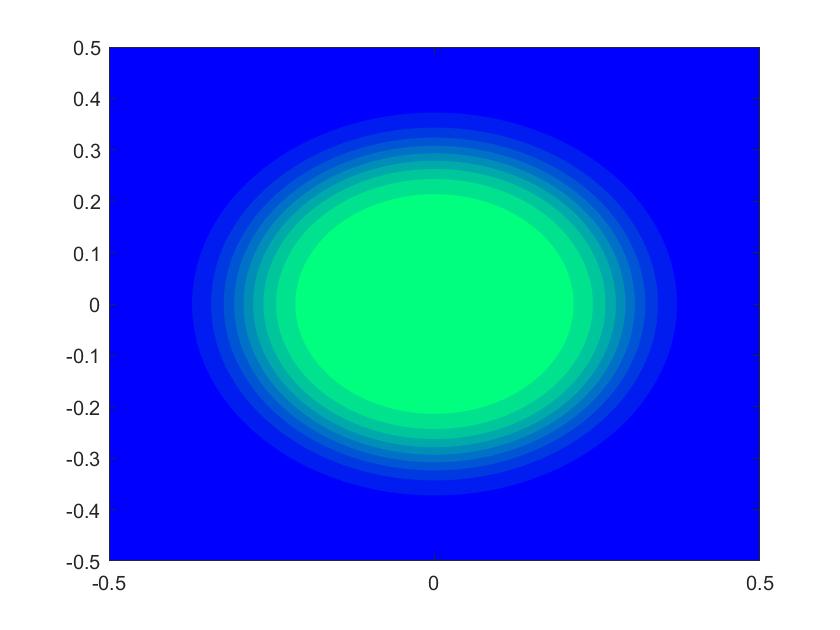}
	}
	
	\subfigure{
		\label{fig5-4}
		\centering
		\includegraphics[width = 120pt,height=120pt]{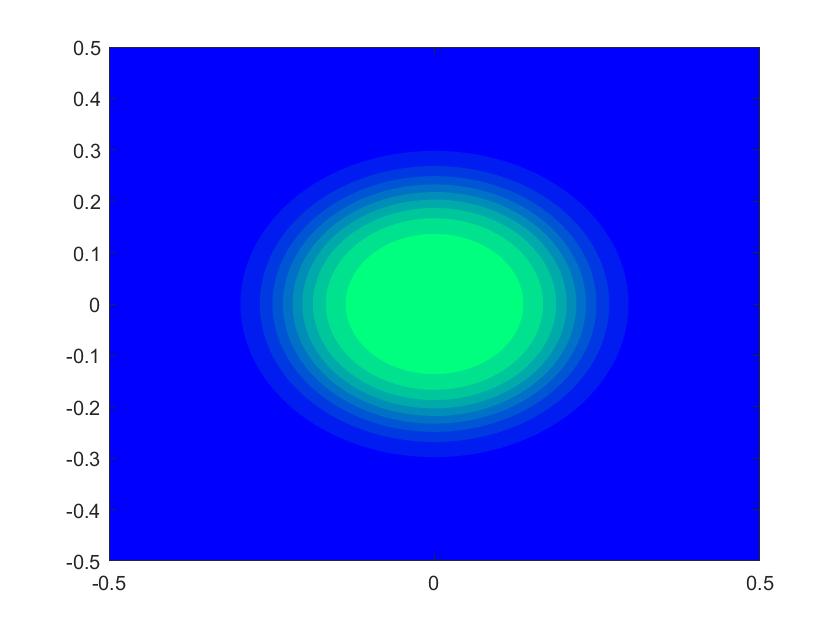}
	}
	\subfigure{
		\label{fig5-5}
		\centering
		\includegraphics[width = 120pt,height=120pt]{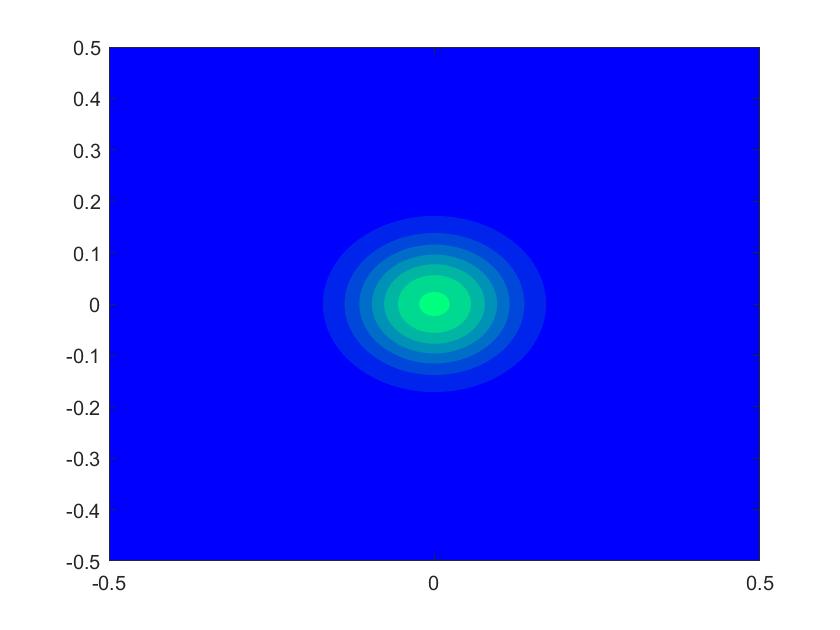}
	}
	\caption{The shrinking circle at time $t=0,1/4 T, 1/2 T,3/4 T, T$ (from left to right and top to bottom) produced by the EIFG2 scheme for Example \ref{ex4} when $d=2$.}
	\label{fig 5}
\end{figure}

\begin{figure}[htbp]
	\centering
	\subfigure{
		\label{fig6-1}
		\centering
		\includegraphics[width = 120pt,height=120pt]{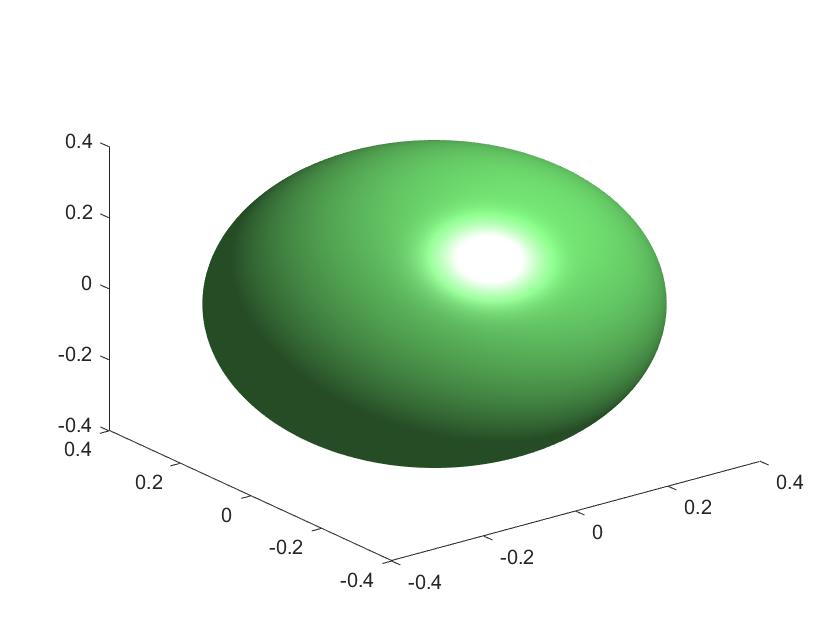}
	}
	\subfigure{
		\label{fig6-2}
		\centering
		\includegraphics[width = 120pt,height=120pt]{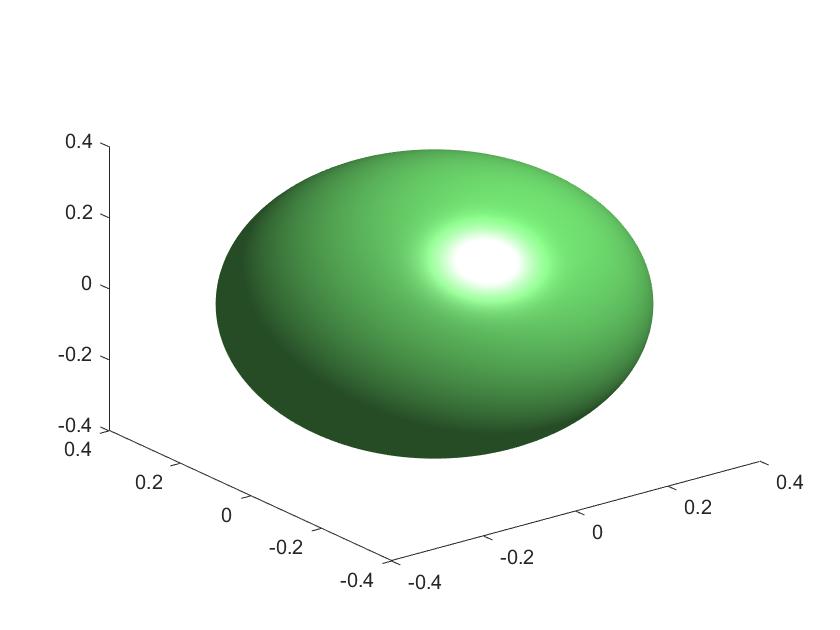}
	}
	\subfigure{
		\label{fig6-3}
		\centering
		\includegraphics[width = 120pt,height=120pt]{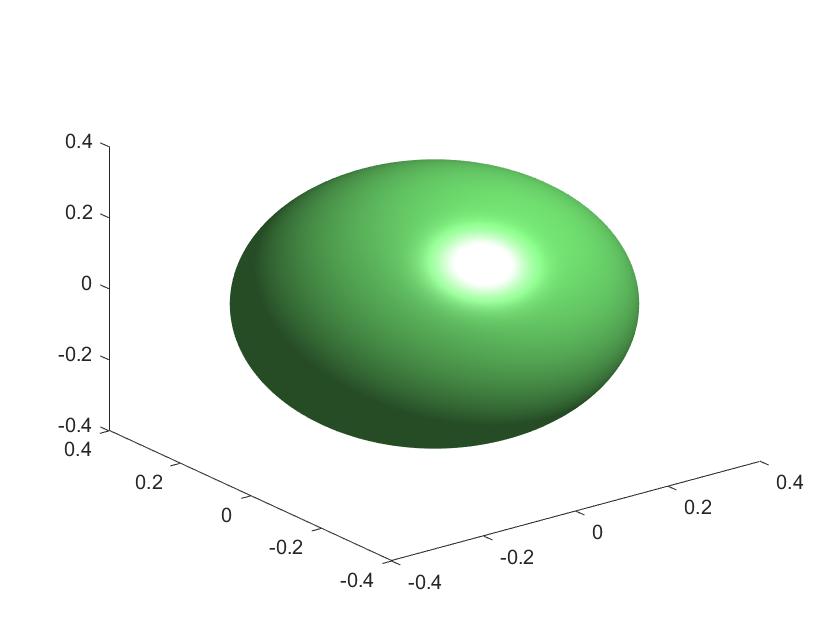}
	}
	
	\subfigure{
		\label{fig6-4}
		\centering
		\includegraphics[width = 120pt,height=120pt]{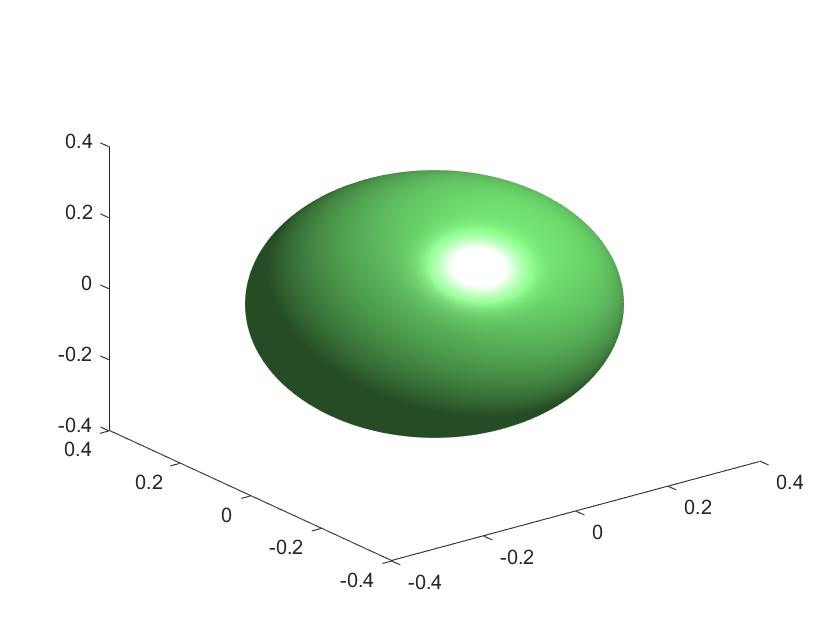}
	}
	\subfigure{
		\label{fig6-5}
		\centering
		\includegraphics[width = 120pt,height=120pt]{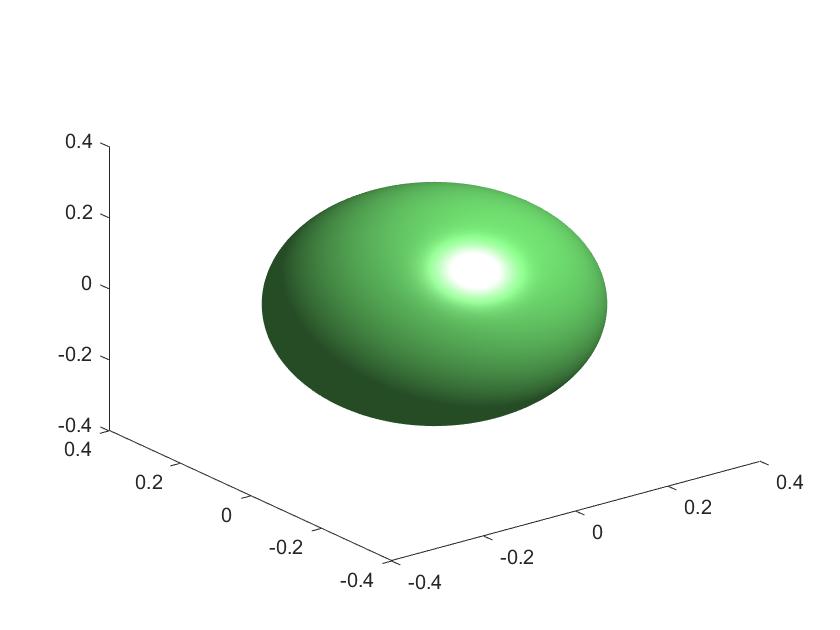}
	}
	\caption{The shrinking sphere at time $t=0,1/4 T, 1/2 T,3/4 T, T$ (from left to right and top to bottom) produced by the EIFG2 scheme for Example \ref{ex4} when $d=3$.}
	\label{fig 6}
\end{figure}

\subsection{3D Burgers equations}
We now present the performance of the proposed EIFG method through numerical simulation of the 3D Burgers equations. The problem we simulate is shown in the following:
\begin{example}
	\label{ex5}
\begin{equation*}
\left\{\begin{split}
&u_t= \epsilon\Delta u -\frac{1}{2} \Big(u^2\Big)_x , \quad \bm x \in \Omega ,\quad \ 0 \leq t \leq T,\\
&u(0, x, y, z) = \frac{2\epsilon\pi\sin(\pi x)}{2+\cos(\pi x)}, \quad \bm x \in \Omega, \\
\end{split}  \right.
\end{equation*}
where $\Omega =[0,2]\times[0,1]\times[0,1],\epsilon=0.1$ and the terminal time is taken to be $T=2$. The exact solution is given by $u(t,x,y,z)=\frac{2\epsilon \pi e^{-\pi^2\epsilon t}\sin(\pi x)}{2+e^{-\pi^2\epsilon t}\cos(\pi x)}$.
\end{example}

For the spatial accuracy tests, we run the EIFG3 scheme with fixed temporal partition $N_T=512$ (i.e. $\Delta T = T/N_T=1/256$) and uniformly refined spatial grids with $N_x\times N_y \times N_z = 512 \times 4\times 4, 1024\times 8 \times 8, 2048\times 16 \times 16, 4096 \times 32 \times 32$. Since the exact solution $u(t,x,y,z)\in H^{\infty}_p$, the numerical error will decay rapidly with the spatial grids refinement and we showed this fact in Figure \ref{fig5}. Meanwhile, for the temporal accuracy tests, we fix the spatial grids with $N_x \times N_y\times N_z=1024 \times 8\times 8$, and the temporal partitions are $N_T=4,8,16,32$ for EIFG2 and $N_T=2,4,8,16$ for EIFG3 method. All numerical results are reported in Table \ref{tab 7}, including the numerical errors measured in $L^2,H^1$ and $H^2$ norms and the corresponding convergence rates. Observing from the numerical results, we can observe the second-order temporal convergence for EIFG2 scheme and third-order temporal convergence for EIFG3 scheme, which coincide well with the error estimates derived in Theorem \ref{thm4}. Furthermore, we present the simulation process in Figure \ref{fig 7}, which are obtained by EIFG3 method with $N_x\times N_y\times N_z = 2048\times 16\times 16$ and $N_T=512$.

\begin{table}[htbp]
	\centering
	\caption{Numerical results on the solution errors measured in the $L^2, H^1$ and $H^2$ norms and corresponding convergence rates for the EIFG2 and EIFG3 schemes in Example \ref{ex1}.}
	\begin{tabular}{|cccccccc|}
		\hline
		$N_T$ & $N_x\times N_y\times N_z$& $\|u_N^n-u(T_n)\|_0$ & CR & $\|u_N^n-u(T_n)\|_1$ & CR & $\|u_N^n-u(T_n)\|_2$ & CR\\
		\hline
		\multicolumn{8}{|c|}{Temporal accuracy tests for EIFG2}\\
		\hline
		4 & $1024 \times 8 \times 8$ & 1.9992e-04 & - & 8.0954e-04 & -& 5.5000e-03& -\\
		8 & $1024 \times 8 \times 8$ & 3.9944e-05 & 2.32 & 1.4434e-04 & 2.49& 9.2581e-04& 2.57\\
		16 & $1024 \times 8 \times 8$ & 9.1424e-06 & 2.13 & 3.1039e-05 & 2.22& 1.8965e-04& 2.29\\
		32 & $1024 \times 8 \times 8$ & 2.2279e-06 & 2.11 & 7.3546e-06 & 2.08& 4.3804e-05& 2.11\\
		\hline
		\multicolumn{8}{|c|}{Temporal accuracy tests for EIFG3}\\
		\hline
		2 & $1024 \times 8 \times 8$ & 1.0423e-04 & - & 3.8960e-04 & - & 3.2000e-03&-\\
		4 & $1024 \times 8 \times 8$ & 1.0587e-05 & 3.30& 3.0752e-05 & 3.66 & 1.8480e-04&4.11\\
		8 & $1024 \times 8 \times 8$ & 9.5445e-07& 3.47 & 2.6088e-06 & 3.56 & 1.3766e-05&3.75\\
		16 & $1024 \times 8 \times 8$ & 1.2234e-07& 2.96& 3.5376e-07 & 2.88 & 1.8923e-06&2.86\\
		\hline
	\end{tabular}
	\label{tab 7}
\end{table}

\begin{figure}[htbp]
	\centering
	\label{fig5}
	\centering
	\includegraphics[width = 150pt,height=150pt]{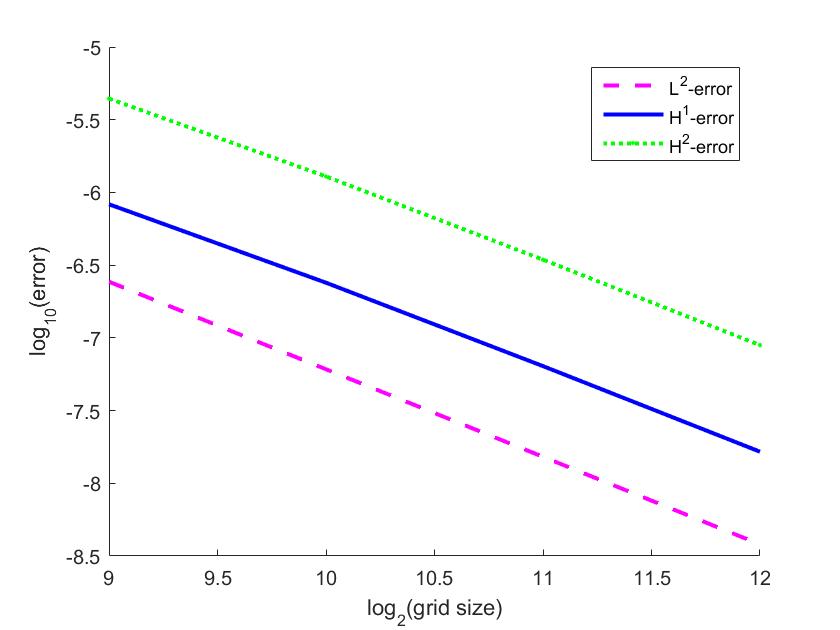}
	\caption{The evolutions of $L^2,H^1$-error and $H^2$-error of the numerical solutions produced by the EIFG3 scheme along with the grid sizes in Example \ref{ex5}.}
\end{figure}

\begin{figure}[htbp]
	\centering
	\subfigure{
		\label{fig7-1}
		\centering
		\includegraphics[width = 120pt,height=120pt]{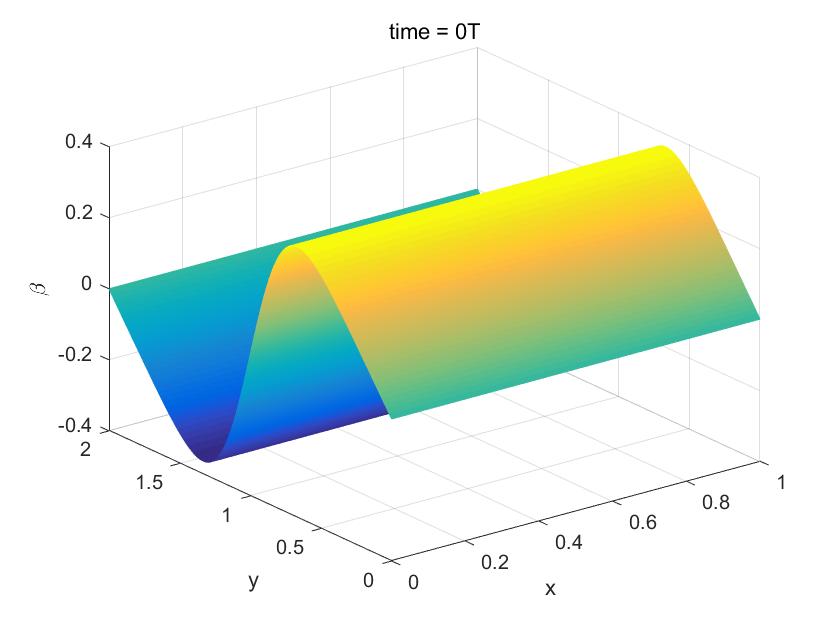}
	}
	\subfigure{
		\label{fig7-2}
		\centering
		\includegraphics[width = 120pt,height=120pt]{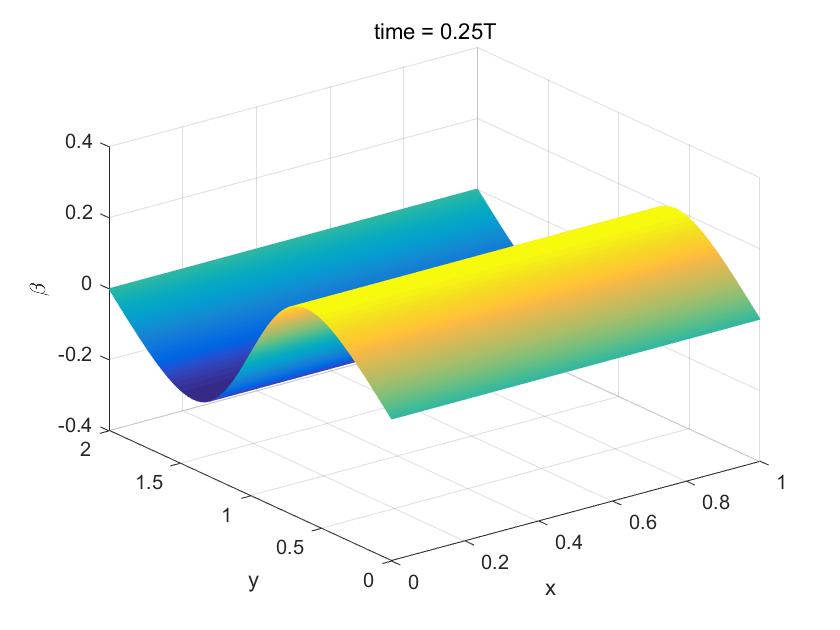}
	}
	\subfigure{
		\label{fig7-3}
		\centering
		\includegraphics[width = 120pt,height=120pt]{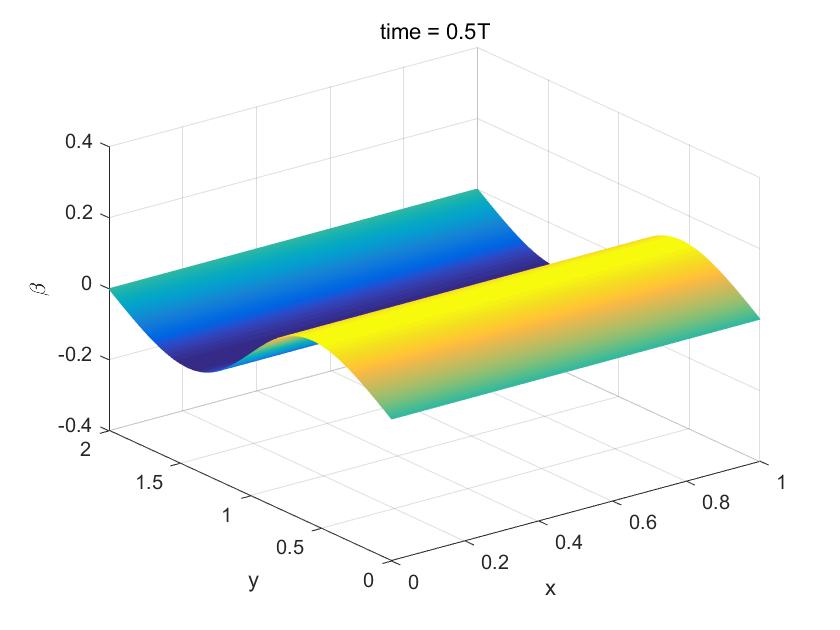}
	}
	
	\subfigure{
		\label{fig7-4}
		\centering
		\includegraphics[width = 120pt,height=120pt]{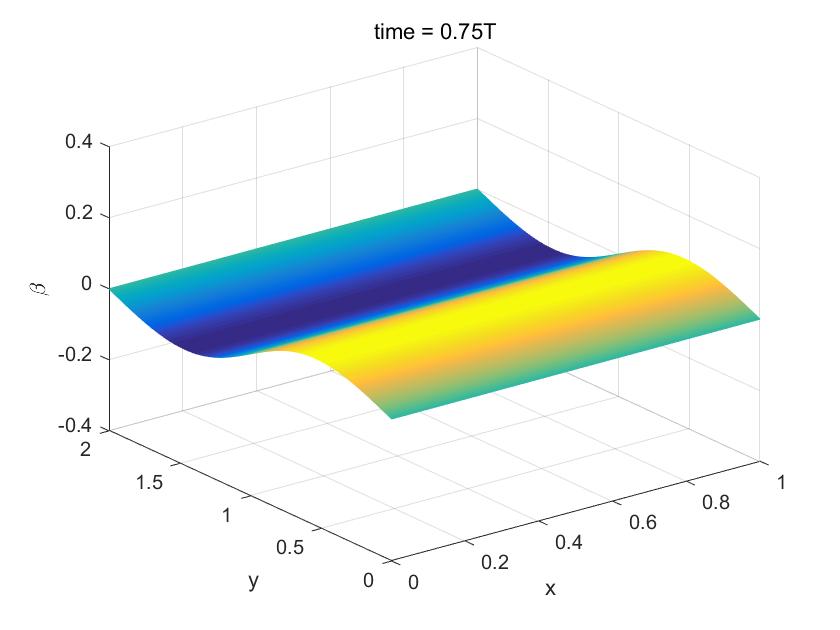}
	}
	\subfigure{
		\label{fig7-5}
		\centering
		\includegraphics[width = 120pt,height=120pt]{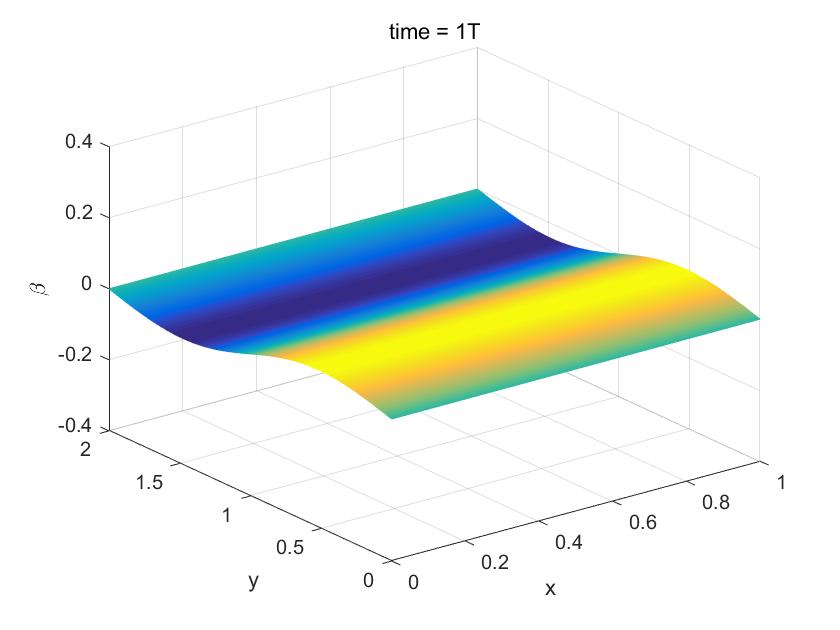}
	}
	\caption{The simulation process at time $t=0,1/4 T, 1/2 T,3/4 T, T$ (from left to right and top to bottom) produced by the EIFG3 scheme for Example \ref{ex5}.}
	\label{fig 7}
\end{figure}

\subsection{3D Grain coarsening simulations}

We now illustrate the performance of the proposed EIFG method through numerical simulation of the 3D grain coarsening process. In particular, we only simulate EIFG2 scheme to obtain 3D phase structures for its solid theoretical support.
\begin{example}
	\label{ex4}
	We consider the grain coarsening process governed by the following 3D Allen-Cahn equation with Flory-Huggins potential function:
	\begin{equation*}
	u_t = \epsilon^2 \Delta u +\frac{\theta}{2}\ln{\frac{1-u}{1+u}}+\theta_c u, \quad (x,y,z)\in \Omega, \quad 0 \leq t \leq T,
	\end{equation*}
	where $\Omega=[0,1]^3$. The initial data is generated by random numbers on each mesh node ranging from -0.9 to 0.9, and the periodic boundary condition is imposed. This equation can be regarded as the $L^2$ gradient flow of the following energy function
	\begin{equation*}
	E(u)=\int_{\Omega}\frac{\theta}{2}\Big((1+u)\ln(1+u)+(1-u)\ln(1-u) \Big)-\frac{\theta_c}{2}u^2+\frac{\epsilon^2}{2}|\nabla u|^2 \dd \bm x
	\end{equation*}
	and thus the energy monotonically decays along the time.
\end{example}

Suppose the interface thickness coefficient $\epsilon = 0.1$ and Flory-Huggins potential parameters $\theta_c=1.6,\theta=0.8$. This problem satisfies the maximum bound principle with the maximum bound value $\gamma \approx 0.9575$, i.e., $|u(t,\bm x)|\leq \gamma$ for all $\bm x \in \Omega$ and $t \geq 0$ \cite{ChenJing2022,LiLi2021}. We run the simulation until $T=20$ with uniform refined spatial grids $N_x =N_y=N_z=128$ ($h=1/128$) and temporal partitions $N_T=8192$ (i.e., $\Delta t=T/N_T=5/2048$). Evolutions of the supremum norm and the energy of the numerical solutions produced by EIFG2 scheme are plotted in Fig \ref{fig 3}. We observe that the maximum bound principle is well preserved and the energy also decays monotonically along the time.

\begin{figure}[htbp]
	\centering
	\subfigure{
		\label{fig3-1}
		\centering
		\includegraphics[width = 120pt,height=120pt]{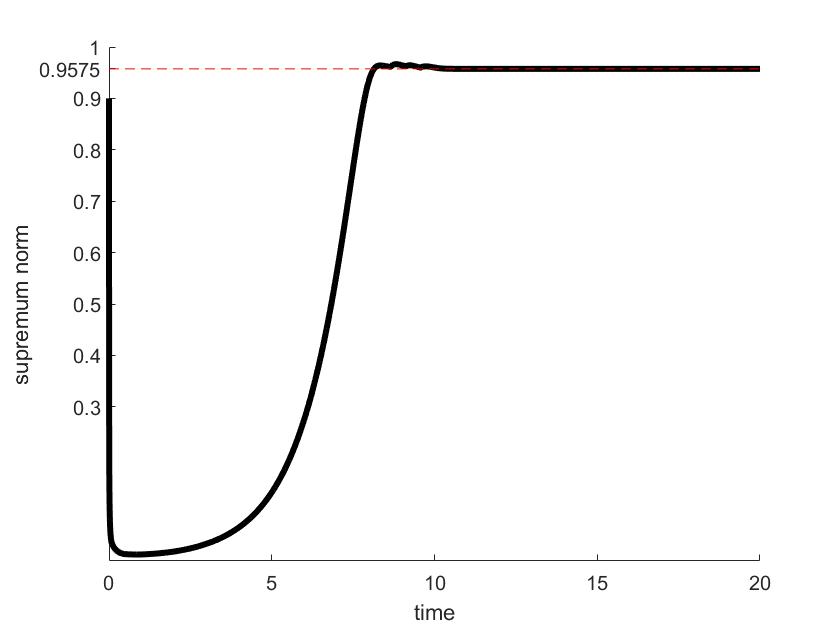}
	}
	\subfigure{
		\label{fig3-2}
		\centering
		\includegraphics[width = 120pt,height=120pt]{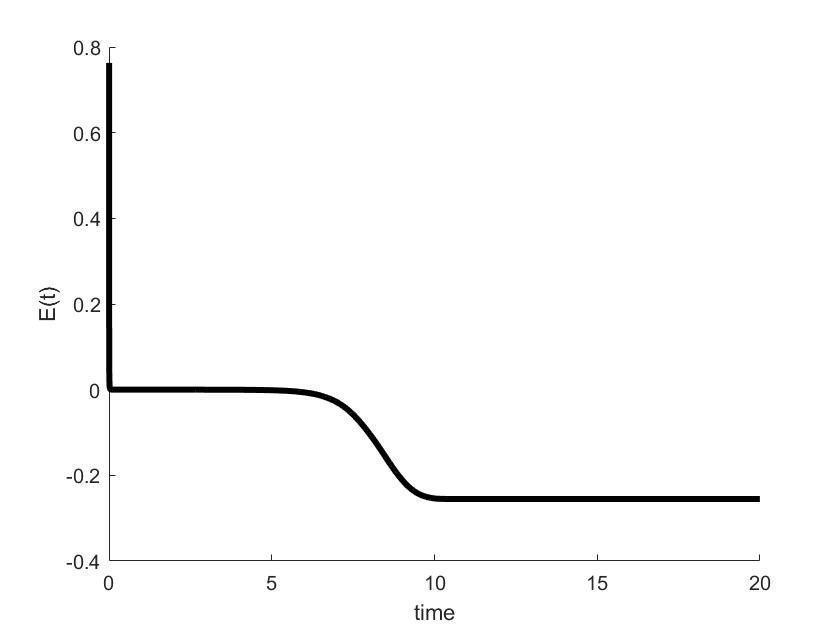}
	}
	\caption{The evolutions of supremum norm (left) and energy (right) of the numerical solution produced by EIFG2 scheme for Example \ref{ex4}.}
	\label{fig 3}
\end{figure}

\section{Conclusions}\label{conclusion}

In this paper, we develop an efficient numerical scheme for solving semilinear parabolic equations taking the form \eqref{eq1-1} in regular domains. The EIFG method is proposed for solving problems with periodic boundary conditions. The fully discrete solution is obtained by first applying Fourier-based (EIFG) Galerkin method for spatial discretization and then explicit exponential Runge-Kutta for temporal integration. We successfully derive optimal error estimates in the $H^2$-norm for EIFG method with two RK stages. Some numerical examples are also presented to illustrate the accuracy and high efficiency of EIFG method. In addition, the numerical method and corresponding error analysis framework developed in this paper also naturally enables us to further investigate the localized ETD methods with solid theoretical support \cite{LiJu2021,HoangJu2018}.

\bibliographystyle{siam}
\bibliography{ref}
\end{CJK}
\end{document}